\pdfoutput=1 

\documentclass[11pt,a4paper,final]{amsart}

\usepackage{a4wide}
\usepackage[utf8]{inputenc}

\usepackage{amsmath}
\usepackage{amsthm}
\usepackage{amssymb}

\usepackage[footnotesize]{caption}
\usepackage{color}
\usepackage{esint}

\usepackage{mathrsfs}

\usepackage{dsfont}

\usepackage{enumitem}

\usepackage[notref,notcite]{showkeys}  

\usepackage[normalem]{ulem}

\usepackage{mathtools}
\mathtoolsset{showonlyrefs}

\usepackage{hyperref}

\DeclareMathOperator{\diam}{diam}

\DeclareMathOperator{\dist}{dist}

\DeclareMathOperator{\lip}{Lip}

\DeclareMathOperator{\spt}{spt}
\DeclareMathOperator{\Hom}{Hom}

\DeclareMathOperator{\im}{im}
\DeclareMathOperator{\lin}{span}

\DeclareMathOperator*{\aplimsup}{ap\,lim\,sup}
\DeclareMathOperator{\ap}{ap}
\DeclareMathOperator{\cnt}{\mathscr{C}}
\DeclareMathOperator{\id}{id}
\DeclareMathOperator{\extalg}{{\textstyle \bigwedge}}

\DeclareMathOperator{\aff}{aff}

\newcommand{\density}{\boldsymbol{\Theta}}

\newcommand{\eqLpnorm}[3]{{(#1)}_{({#2})}({#3})}
\newcommand{\esssups}[1]{({#1}) \operatorname*{ess\,sup}}
\newcommand{\esssup}[1]{\bigl({#1}\bigr) \operatorname*{ess\,sup}}

\newcommand{\AD}[1]{\operatorname{AD}({#1})}
\newcommand{\without}{\sim}
\newcommand{\simp}{\bigtriangleup}

\newcommand{\interior}[1]{\operatorname{Int}{#1}}

\newcommand{\Tan}[2]{\operatorname{Tan}({#1},{#2})}
\newcommand{\apTan}[3]{\operatorname{Tan}^{#1}({#2},{#3})}

\newcommand{\tfint}[2]{{\textstyle\fint_{#1}^{#2}}}

\newcommand{\unitmeasure}[1]{\boldsymbol{\alpha}(#1)}

\newcommand{\ud}{\ensuremath{\,\mathrm{d}}}

\newcommand{\restrict}{ \mathop{ \rule[1pt]{.5pt}{6pt} \rule[1pt]{4pt}{0.5pt} }\nolimits }

\newcommand{\grass}[2]{\mathbf{G}\left(#1,#2\right)}

\newcommand{\OP}{\mathbf{O}^\ast}

\newcommand{\proj}[1]{#1_\natural}

\newcommand{\tbcup}{{\textstyle \bigcup}}

\newcommand{\pproj}[1]{{#1}_\natural^\perp}

\newcommand{\lIm}{[}
\newcommand{\rIm}{]}

\newcommand{\tcup}{\textstyle{\bigcup}}

\newcommand{\R}{\mathbb{R}}

\newcommand{\N}{\mathbb{N}}

\newcommand{\kav}{\mathcal{K}}

\newcommand{\mlc}{\kappa}

\newcommand{\HM}{\mathcal{H}}

\newcommand{\LM}{\mathcal{L}}

\newcommand{\oball}[2]{\mathbf{U}({#1},{#2})}

\newcommand{\cball}[2]{\mathbf{B}({#1},{#2})}

\newcommand{\cyl}[2]{\mathbf{C}({#1},{#2})}

\newcommand{\pp}{\mathfrak{p}}
\newcommand{\qq}{\mathfrak{q}}

\newcommand{\hmin}{h_{\min}}

\definecolor{mygreen}{RGB}{50,70,0}

\swapnumbers
\theoremstyle{plain}
\newtheorem{thm}{Theorem}[section]
\newtheorem{lem}[thm]{Lemma}
\newtheorem{prop}[thm]{Proposition}
\newtheorem{cor}[thm]{Corollary}

\theoremstyle{remark}
\newtheorem{rem}[thm]{Remark}
\newtheorem*{rem*}{Remark}

\theoremstyle{definition}
\newtheorem{defin}[thm]{Definition}

\setlist[enumerate,1]{label=(\alph*), ref=(\alph*), leftmargin=*}

\author{S{\l}awomir Kolasi{\'n}ski}
\address{Max Planck Institute for Gravitational Physics
  (Albert Einstein Institute)\\
  Am~M{\"u}hlen\-berg~1, D-14476 Golm\\
  Germany}
\email{skola@mimuw.edu.pl}


\date{\today}

\title{Higher order rectifiability of measures via averaged discrete curvatures}

\begin{document}

\begin{abstract}
    We provide a sufficient geometric condition for~$\R^n$ to be countably
    $(\mu,m)$ rectifiable of class~$\cnt^{1,\alpha}$ (using the terminology of
    Federer), where $\mu$ is a Radon measure having positive lower density and
    finite upper density $\mu$ almost everywhere. Our condition involves
    integrals of certain many-point interaction functions (discrete curvatures)
    which measure flatness of simplices spanned by the parameters.
\end{abstract}

\maketitle

\section{Introduction}
\label{sec:intro}

Let $\HM^m$ denote the $m$~dimensional Hausdorff measure over~$\R^n$. A~measure
$\mu$ is said to be \emph{$\AD m$ regular} if there exists $A \in [1,\infty)$
such that $A^{-1} \le r^{-m} \mu(\cball xr) \le A$ for $x \in \spt \mu$ and $r
\in (0,1)$. We say that $\mu$ is~\emph{uniformly rectifiable} if it has the
\emph{big pieces of Lipschitz images} property, which means that it is $\AD m$
regular and there exist $\theta, M \in (0,\infty)$ such that for $x \in \spt
\mu$ and $r \in (0,1)$ there exists a Lipschitz map $f : \cball 0r \cap \R^m \to
\R^n$ such that $\lip f < M$ and $\mu(\cball xr \cap f\lIm \cball 0r \rIm) \ge
\theta r$. David and Semmes~\cite{DS91,DS93} were studying uniformly rectifiable
sets $\Sigma \subseteq \R^n$ in the context of harmonic analysis and in the
search for a~geometric criterion yielding boundedness of certain singular
integral operators on~$L^2(\HM^m \restrict \Sigma)$. To characterize these sets,
they introduced the \emph{beta-numbers} defined for $p \in [1,\infty)$, $x \in
\R^n$, $r \in (0,\infty)$ as follows
\begin{displaymath}
    \beta^m_{\mu,p}(x,r) = r^{-1} \inf_{L}
    \bigl( r^{-m} {\textstyle \int_{\cball xr}} \dist(y,L)^p \ud \mu(y) \bigr)^{1/p} \,,
\end{displaymath}
where the infimum is taken with respect to all affine $m$~dimensional planes~$L$
in~$\R^n$. For $p,q \in [1,\infty)$ and $B \subseteq \R^n$ Borel we set
\begin{displaymath}
    J_{\mu,p,q}(B) =
    {\textstyle \int_{B} \int_0^{\diam B} \beta^m_{\mu,q}(x,r)^p \frac{\ud r}{r} \ud \mu(x)} \,.
\end{displaymath}
Following~\cite[Definition~1.2 on p.~313]{DS93} we say that $\mu$ satisfies the
\emph{$(p,q)$ geometric lemma} if $\mu$ is $\AD m$ regular and there exists some
$C \in (0,\infty)$ such that $J_{\mu,p,q}(B) \le C \mu(B)$ for all balls $B
\subseteq \R^n$. By~\cite[p.~22]{DS93} measures which satisfy the $(2,2)$
geometric lemma are uniformly rectifiable.

If $T = (a_0,\ldots,a_{m+1}) \in (\R^n)^{m+2}$, let~$\simp T$ be the convex hull
of the set $\{a_0,\ldots,a_{m+1}\}$. For $B \subseteq \R^n$ Borel we define the
\emph{integral Menger curvature} of a measure $\mu$ on~$B$ by
\begin{displaymath}
    \mathcal M_{\mu}(B) = 
    {\textstyle \int_{B^{m+2}} \frac{\HM^{m+1}(\simp(a_0,\ldots,a_{m+1}))^2}{\diam(\{a_0,\ldots,a_{m+1}\})^{(m+2)(m+1)}}
    \ud \mu^{m+2}(a_0,\ldots,a_{m+1})} \,.
\end{displaymath}
This was (up to a constant) one of the functionals considered by Lerman and
Whitehouse in~\cite{LW09,LW11} were they showed that if $\mu$ is $\AD
m$~regular, then~$J_{\mu,2,2}$ is comparable to~$\mathcal M_{\mu}$ on balls
establishing a~new characterization of uniform rectifiability and providing a
connection between the beta-numbers and Menger-type curvatures.
In~\cite[Corollary~6.1]{LW09} a~characterization of measures satisfying the
$(p,p)$~geometric lemma for $1 \le p < \infty$ is also given.

Following Federer~\cite[3.2.14]{Fed69} and Anzellotti and Serapioni~\cite{AS94}
we say that a set $E \subseteq \R^n$ is \emph{countably $(\mu,m)$~rectifiable of
  class~$\cnt^{1,\alpha}$} if there exists a countably family $\mathcal A$ of
$m$~dimensional submanifolds of~$\R^n$ of class~$\cnt^{1,\alpha}$ such that
$\mu(E \without \tbcup \mathcal A) = 0$. Let $\density_*^m(\mu,a)$ and
$\density^{m*}(\mu,a)$ denote the lower an upper densities of~$\mu$ at $a \in
\R^n$ as defined in~\cite[2.10.19]{Fed69}.

Quite recently Meurer~\cite{Meu15a} proved, assuming a~priori that $\Sigma
\subseteq \R^n$ is merely~\emph{Borel} and $\mu = \HM^m \restrict \Sigma$, that
if $\mathcal M_{\mu}(\R^n) < \infty$, then $\Sigma$ is countably
$(\HM^m,m)$~rectifiable of class~$\cnt^{1}$. His work can be seen as a~higher
dimensional counterpart of the result of David~\cite{Dav98}
and~L{\'e}ger~\cite{Leg99} done in connection with the famous Vitushkin's
conjecture on removable sets for bounded analytic functions. Azzam and
Tolsa~\cite{AT15} and Tolsa~\cite{Tol15} proved that if $\mu$ is Radon and
satisfies $0 < \density^{m*}(\mu,a) < \infty$ for $\mu$ almost all~$a$, then
$\R^n$ is countably $(\mu,m)$~rectifiable of class~$\cnt^1$ \emph{if and only
  if} $\int_0^{1} \beta^m_{\mu,2}(x,r)^2 \frac{\ud r}{r} < \infty$ for $\mu$
almost all~$a$. However, in this generality the quantities considered
in~\cite{Meu15a} and in~\cite{AT15,Tol15} are not known to be directly
comparable as was the case in the $\AD m$ regular case due to~\cite{LW09,LW11}.


Higher order rectifiability, mostly of functions
(see~\cite[Definition~2.5]{AS94}), has also been studied for a long time.
In~the context of functions it is rather called a \emph{Lusin-type
  approximation}. Calder{\'o}n and Zygmund~\cite[Theorems 9 and 13]{CZ61},
Re{\v{s}}etnjak~\cite{Res68} and, more recently, Liu and Tai~\cite{LT94}, Lin
and Liu~\cite{LL13} gave conditions for higher order rectifiability of functions
in terms of existence of approximating polynomials at almost all points.
The~classical Alexandrov's theorem (see~\cite{Ale39} or~\cite{Fu11}) and its
generalization by Alberti~\cite{Alb94} can be seen as a~$\cnt^2$ rectifiability
result for convex functions.

Higher order rectifiability is an important feature of sets in geometric
analysis. It was observed by Sch{\"a}tzle~\cite[\S3]{Sch09} that it can be used
for proving regularity of sets governed by a~PDE. This philosophy was employed
later by Menne~\cite{Men09,Men10,Men11,Men12} and Menne and the
author~\cite{KM15} for showing certain regularity results for varifolds.

In the present article we provide a sufficient condition for rectifiability of
class~$\cnt^{1,\alpha}$ in terms of functionals similar to $\mathcal M_{\mu}$.
Whenever $\mu$ is a measure over~$\R^n$, and $l \in \{1,2,\ldots,m+2\}$, and
$\alpha \in [0,1]$, and $p \in [1,\infty)$, and $a, a_0, \ldots, a_{m+1} \in
\R^n$, and $T = (a_0,\ldots,a_{m+1})$, and $r \in (0,\infty]$ we set
\begin{gather}
    \mlc(T) = \frac{\HM^{m+1}(\simp T)}{\diam(\simp T)^{m+1}} 
    \quad \text{if $\diam(\simp T) > 0$} \,,
    \quad
    \mlc(T) = 0 \quad \text{otherwise} \,,
    \\
    \mlc^{l,p,\alpha}_{\mu,a,r}(a_1,\ldots,a_{l-1}) = 
    \esssup{\mu^{m+2-l}}_{a_l,\ldots,a_{m+1} \in \cball ar}
    \frac{\mlc(a,a_1,\ldots,a_{m+1})^p}{\diam(\{a,a_1,\ldots,a_{m+1}\})^{m(l-1) + \alpha p}} \,,
    \\
    \label{eq:def-kav}
    \kav^{l,p,\alpha}_{\mu}(a,r) 
    = \int_{\cball ar} \cdots \int_{\cball ar} 
    \mlc^{l,p,\alpha}_{\mu,a,r}(b_1,\ldots,b_{l-1})
    \ud \mu(b_1) \cdots \ud \mu(b_{l-1})
\end{gather}
with the understanding that there is no essential supremum in case $l = m+2$ and
there is no integral in case $l=1$.

If $p > m(l-1)$, and $\alpha = 1 - m(l-1)/p$, and $\mu = \HM^m \restrict \Sigma$
for some compact set $\Sigma \subseteq \R^n$, then $\mathcal E(\mu) = \int
\kav^{l,p,\alpha}_{\mu}(a,\infty) \ud \mu(a)$ coincides with one of the
functionals analyzed by Strzelecki and von der Mosel~\cite{SvdM11a}, Blatt and
the author~\cite{BK12}, Szuma{\'n}ska and the author~\cite{KolSzum}, Strzelecki
and von der Mosel and the author~\cite{KSM13,KSM15}, and by the
author~\cite{Kol15a}. Analogues of the Morrey-Sobolev embedding theorem for
sets, where $\mlc$ plays roughly the role of the second weak derivative, were
studied in~\cite{SvdM11a,Kol15a,KolSzum}. Geometric characterizations of graphs
of certain (fractional) Sobolev maps were given in~\cite{BK12,KSM13}.
In~\cite{KSM15} the authors use $\mathcal E$ to solve geometric variational
problems with topological constraints. In~\cite{BK12,KSM13} a full geometric
characterization of graphs of some (fractional) Sobolev maps is established.

If~$l = m+2$, and $p = 2$, and $\alpha = 0$, then $\mathcal E(\mu)$ equals
$\mathcal M_{\mu}(\R^n)$ which was considered in~\cite{LW09,LW11,Meu15a}.
Formally, if one sets $l = m+2$, and $p = 2$, and $\alpha = m - (m+2)(m+1)/2$,
then $\mathcal E$ yields also one of the quantities used in~\cite[\S4]{LW12} to
approximate the least square error of a measure.

Our main result reads as follows.
\begin{thm}
    \label{thm:main}
    Let $\mu$ be a Radon measure over $\R^n$ such that
    \begin{equation}
        \label{eq:density-bounds}
        0 < \density_*^m(\mu,x) \le \density^{m*}(\mu,x) < \infty
        \quad \text{for $\mu$ almost all $x$} \,,
    \end{equation}
    and $l \in \{1,2,\ldots,m+2\}$, and $\alpha \in (0,1]$, and $p \in
    [1,\infty)$. Assume $\kav^{l,p,\alpha}_{\mu}(a,1) < \infty$ for $\mu$ almost
    all~$a$. Then $\R^n$ is countably $(\mu,m)$~rectifiable of
    class~$\cnt^{1,\alpha}$ and $\mu$ is absolutely continuous with respect
    to~$\HM^m$.

    Moreover, if $\alpha < 1$, then for any $\varepsilon \in (0,1-\alpha)$ there
    exists a measure $\mu$ satisfying~\eqref{eq:density-bounds} and
    $\kav^{l,p,\alpha}_{\mu}(a) < \infty$ for $\mu$~almost all~$a$ and such that
    $\R^n$~is not countably $(\mu,m)$~rectifiable of
    class~$\cnt^{1,\alpha+\varepsilon}$.
\end{thm}

The converse of Theorem~\ref{thm:main} does not hold due to the
example~\cite{KolSzum}. In view of the characterization of graphs of functions
of Sobolev-Slobodeckij class~$\cnt^1 \cap W^{1+\alpha,p}$ obtained
in~\cite{BK12} one could expect that finiteness $\mu$~almost everywhere of
$\kav^{l,p,\alpha}_{\mu}$ should rather characterize ``rectifiability of
class~$W^{1+\alpha,p}$'' -- a~notion not yet defined.

We prove the first part of~\ref{thm:main} in section~\ref{sec:hor-menger}. First
we use standard methods of geometric measure theory to reduce the problem to the
case when $\mu$ is roughly $\AD m$ regular, which is possible because we assume
$0 < \density_*^m(\mu,x)$ for $\mu$ almost all~$x$. Then, for $\mu$ almost all
points~$x$, we find $m$~dimensional planes that approximate the measure in
smaller and smaller scales around~$x$ and, due to the condition $\alpha > 0$,
we~prove that they converge in the Grassmannian to some plane which must contain
the approximate tangent cone of~$\mu$ at~$x$ -- this is the heart of the proof;
see~\ref{lem:height-control}. From there we conclude using~\cite[2.8(5)]{All72}
that $\R^n$~is countably $(\mu,m)$~rectifiable of class~$\cnt^1$. This allows to
reduce the problem farther to the case when $\mu = \HM^m \restrict \Sigma$,
where $\Sigma$ is a subset of a graph of some $\cnt^1$ function. Next, we use
the decay rates obtained in~\ref{lem:height-control} together
with~\ref{cor:main-rect-cond}, which is a corollary of~\cite[Lemma A.1]{Sch09},
to get the conclusion.

The proof of the second part of~\ref{thm:main} is contained in
section~\ref{sec:sharp}. It bases on the fact that for any $0 \le \alpha < \beta
\le 1$ there exists a set which is a graph of some $\cnt^{1,\alpha}$ map but is
not countably $(\HM^m,m)$ rectifiable of class~$\cnt^{1,\beta}$ which can be
deduced from~\cite[Appendix]{AS94}. In section~\ref{sec:sharp} we also give
examples of other functions that could be used in~\ref{thm:main} in place
of~$\mlc$.

Our result can be seen as an extension of~\cite{Meu15a} but is \emph{not}
stronger. A~glance at the outline of the proof given above reveals why our
problem is extremely simpler than that considered in~\cite{Meu15a}
or~\cite{AT15}. The first reason is that we assume $\alpha > 0$ which
immediately gives convergence of the approximating planes
in~\ref{lem:height-control}. The second, but actually the crucial one, is that
we assume $\density_*^m(\mu,a) > 0$ for $\mu$~almost all~$a$ which allows to
reduce, roughly, to the $\AD m$ regular case.

\section{Notation}
\label{sec:notation}
In principle we shall use the book of Federer~\cite{Fed69} as our main reference
and source of definitions, and we shall adopt some, but not all, of its
notation. In~particular we shall write $\{ x \in X : P(x) \}$, in contrast to $X
\cap \{ x : P(x) \}$, for the set of those $x \in X$ which satisfy some
predicate $P$. We also prefer to say that a function is ``injective'' rather
than ``univalent''. The symbols $\R$ and $\N$ shall be used for the set of real
and natural numbers including zero. Moreover, whenever $s,t \in \R \cup
\{-\infty,\infty\}$ and $s < t$, we shall write $(s,t)$, $[s,t]$ for the open
and closed intervals in~$\R$ and also $(s,t]$ and $[s,t)$ with the usual
meaning. If $A$ and $B$ are sets, we write $A \without B$ for the set theoretic
difference. If $X$ is a vector space, $A,B \subseteq X$, $c \in X$ and $r \in
(0,\infty)$ we adopt the notation $c + A = \{ c + a : a \in A \}$, $rA = \{ r a
: a \in A \}$ and $A + B = \{ a+b : a \in A \,,\ b \in B \}$. When we write
$\R^n$ we always mean the $n$ dimensional Euclidean space with the standard
scalar product denoted $u \bullet v$ for $u,v \in \R^n$. We~shall write $\oball
ar$ and $\cball ar$ for the open and closed ball centered at~$a$ and of
radius~$r$ in the metric space to which~$a$ belongs to. We~adopt the definition
of a~\emph{measure} from~\cite[2.1.2]{Fed69}, which is sometimes called
an~\emph{outer measure} in the literature. The~symbols $\HM^m$ and $\LM^m$ stand
for the $m$~dimensional Hausdorff and Lebesgue (outer) measures as defined
in~\cite[2.10.2(1) and 2.6.5]{Fed69}. Whenever $m \in \N \without \{0\}$ we use
the symbol $\unitmeasure{m}$ for the Lebesgue measure of the unit ball
in~$\R^m$. If~$X$ and~$Y$ are normed vector spaces, $U \subseteq X$ is open, $k
\in \N$, and $\alpha \in [0,1]$, then a~function $f : U \to Y$ is said to be of
class $\cnt^{k,\alpha}$ if $f$ is continuous, has continuous derivatives up to
order $k$ (cf.~\cite[3.1.1\,,\,3.1.11]{Fed69}), and the $k^{\text{th}}$ order
derivative $D^kf$ satisfies the H{\"o}lder condition with exponent~$\alpha$
(cf.~\cite[5.2.1]{Fed69}); in this case we write $f \in \cnt^{k,\alpha}(U,Y)$.
The~image of a set $A \subseteq X$ under a mapping $f : X \to Y$ is denoted
$f\lIm A \rIm$ and similarly $f^{-1}\lIm B\rIm$ denotes the preimage of a set $B
\subseteq Y$. We write $\id_X$ for the identity function on~$X$. Whenever $X$ is
a metric space, $A \subseteq X$, and $x \in X$, we use the notation $\dist(x,A)$
for the distance of~$x$ from~$A$. We write $A^l$ to denote the Cartesian product
of~$l \in \N \without \{0\}$ copies of a set~$A$ and if $f : A \to B$, then $f^l
: A^l \to B^l$ is the Cartesian product of $l$~copies of~$f$, i.e.,
$f^l(a_1,\ldots,a_l) = (f(a_1),\ldots,f(a_l))$. Similarly, $\mu^l$ shall denote
the product of $l$~copies of a measure~$\mu$ (cf.~\cite[2.6.1]{Fed69}). For the
Grassmannian of $m$~dimensional planes in $\R^n$ we write $\grass nm$
(cf.~\cite[1.6.2]{Fed69}). With each $P \in \grass nm$ we associate the
orthogonal projection $\proj P : \R^n \to P \subseteq \R^n$ onto $P$ and the
orthogonal complement $P^{\perp} = \ker \proj P$. Whenever $v_1,\ldots,v_k$ are
vectors in some vector space $X$, we write $\lin\{v_1,\ldots,v_l\}$ for the
linear span of these vectors. If $\mu$ measures some set $X$, $f : X \to \R$ is
$\mu$-measurable and $A \subseteq X$ is $\mu$-measurable, we write $\fint_A f
\ud \mu = \mu(A)^{-1} \int_A f \ud \mu$ for the \emph{mean value} of~$f$
on~$A$. By~$X \ni x \mapsto f(x)$ we mean an \emph{unnamed function} with domain
$X$ mapping~$x \in X$ to~$f(x)$. For the \emph{essential supremum} of a function
$f : X \to \R$ with respect to a~measure~$\mu$ over~$X$ we write
$\esssups{\mu}(f)$, which is defined to be equal to $\eqLpnorm{\mu}{\infty}{f}$
in the notation of~\cite[2.4.12]{Fed69}. To optimize space we shall sometimes
write $\esssups{\mu}_{x \in X}(f(x))$ instead of $\esssups{\mu}(X \ni x \mapsto
f(x))$.

The reader might want to recall the definitions of: the space of orthogonal
projections $\OP(n,m)$~\cite[1.7.4]{Fed69}, the exterior algebra $\extalg_* X$
of a vector space $X$ with its associated wedge
product~$\wedge$~\cite[1.3]{Fed69}, tangent cone $\Tan Sx$ of a set $S \subseteq
\R^n$~\cite[3.1.21]{Fed69} at~$x \in \R^n$, approximate $m$-tangent cone
$\apTan{m}{\mu}{x}$ of a measure $\mu$~\cite[3.2.16]{Fed69} at~$x \in \R^n$.

Sometimes we write ``$\Gamma_{x.y}$'' to denote the number that appeared under
the name~''$\Gamma$'' in the formulation of~x.y. Throughout the paper $n$ and
$m$ shall denote two integers satisfying $1 \le m < n$.

\section{Approximate tangent cones}

\begin{rem}
    \label{rem:ap-tangent-vector}
    For $a,v \in \R^n$, $\varepsilon \in (0,\infty)$ define the cone
    \begin{displaymath}
        \mathbf{E}(a,v,\varepsilon) 
        = \bigl\{
            b \in \R^n : \exists\,t \in (0,\infty) \ |t (b-a) - v| < \varepsilon
        \bigr\} \,.
    \end{displaymath}
    Then (cf.~\cite[3.2.16]{Fed69}) $v \in \apTan{m}{\mu}{a}$ if and only if
    \begin{displaymath}
        \density^{*m}(\mu \restrict \mathbf{E}(a,v,\varepsilon)) > 0
        \quad \text{for all $\varepsilon \in (0,\infty)$.}
    \end{displaymath}

    Note that, if $0 < \varepsilon < |v|$, then $b \in
    \mathbf{E}(a,v,\varepsilon)$ if and only if
    \begin{displaymath}
        b \ne a 
        \quad \text{and} \quad
        \frac{b-a}{|b-a|} \bullet \frac{v}{|v|} > \left( 1 - \frac{\varepsilon^2}{|v|^2} \right)^{1/2} \,.
    \end{displaymath}
\end{rem}

The notion of an approximate tangent cone $\apTan{m}{\mu}{a}$
(cf.~\cite[3.2.16]{Fed69}) that we use is different from the notion of the
tangent space defined by blow-ups in~\cite[11.2]{Sim83}. However, if we have the
following.

\begin{prop}
    \label{prop:apTan-subset-T}
    Suppose $\mu$ is a Radon measure over~$\R^n$, and $T \in \grass nm$, and $a \in
    \R^n$, 
    \begin{equation}
        \label{eq:apTan-T-assum}
        \text{and}  \quad
        \lim_{r \downarrow 0} r^{-m} \int_{\cball ar} \frac{|\pproj T (b-a)|}{|b-a|} \ud \mu(b) 
        = 0 \,.
    \end{equation}
    Then $\apTan{m}{\mu}{a} \subseteq T$. 
\end{prop}

\begin{proof}
    If $\apTan{m}{\mu}{a} \without\{0\} = \varnothing$, the conclusion is
    evident. In all other cases we shall prove the proposition by
    contradiction. If there existed $v \in \apTan{m}{\mu}{a} \without T$; then
    $|\pproj Tv| > 0$. Recalling~\ref{rem:ap-tangent-vector}, if $\varepsilon
    \in (0,\infty)$ satisfied $\varepsilon < \tfrac 14|\pproj Tv|$, then for
    each $b \in \mathbf{E}(a,v,\varepsilon)$ setting $t = (b-a) \bullet v
    |b-a|^{-2}$, we would have
    \begin{multline*}
        \left| \frac{b-a}{|b-a|} \bullet \frac{v}{|v|} \right| \cdot
        \left| \pproj T \frac{b-a}{|b-a|} \right|
        = \frac{| \pproj T t(b-a) |}{|v|}
        \ge \frac{ | \pproj T v |  - | \pproj T (v - t(b-a)) | }{|v|}
        \\
        \ge \frac{ | \pproj T v | - \varepsilon }{|v|}
        \ge  \frac 34 \frac{| \pproj T v |}{|v|} > 0 \,.
    \end{multline*}
    Hence, for any $r > 0$ and $\varepsilon \in \bigl(0, \frac 14 |\pproj T
    v|\bigr)$, we would obtain
    \begin{multline}
        \label{eq:apTan-contr-est}
        r^{-m} \int_{\cball ar} \frac{|\pproj T (b-a)|}{|b-a|} \ud \mu(b)
        \ge r^{-m} \int_{\mathbf{E}(a,v,\varepsilon) \cap \cball ar}
        \frac{|\pproj T (b-a)|}{|b-a|} \ud \mu(b)
        \\
        \ge \frac 34
        \frac{| \pproj{T(x)} v |}{|v|}
        \left( 1 - \frac{\varepsilon^2}{|v|^2} \right)^{-1/2} 
        \frac{\mu(\mathbf{E}(a,v,\varepsilon) \cap  \cball ar)}{r^{m}} \,.
    \end{multline}
    Since we assumed $v \in \apTan{m}{\mu}{a}$, we could argue that
    $\density^{*m}(\mu \restrict \mathbf{E}(a,v,\varepsilon),a) > 0$ for all
    $\varepsilon \in (0,\infty)$. Then, for $\varepsilon \in (0, \frac 14
    |\pproj T v|)$, taking $\limsup_{r \downarrow 0}$ on both sides
    of~\eqref{eq:apTan-contr-est}, we would get
    \begin{gather*}
        \limsup_{r \downarrow 0} r^{-m} \int_{\cball ar} 
        \frac{|\pproj T (b-a)|}{|b-a|} \ud \mu(b) > 0 \,,
    \end{gather*}
    which is impossible due to the assumption~\eqref{eq:apTan-T-assum}. Thereby,
    we conclude that it was not possible to choose $v \in \apTan{m}{\mu}{a}
    \without T$; thus $\apTan{m}{\mu}{a} \subseteq T$.
\end{proof}

\begin{rem}
    \label{rem:r-limit}
    Observe that
    \begin{equation}
        \label{eq:weaker-int-limit}
        \lim_{r \downarrow 0} r^{-m-1} \int_{\cball ar} |\pproj T (b-a)| \ud \mu(b) 
        = 0 
    \end{equation}
    implies \eqref{eq:apTan-T-assum} which can be verified by representing the
    integral over $\cball ar$ by a series of integrals over ``annuli'' $\cball
    a{2^{-k}r} \without \oball a{2^{-k-1}r}$ for $k \in \N$. Hence, the
    conclusion of~\ref{prop:apTan-subset-T} holds also with
    assumption~\eqref{eq:apTan-T-assum} replaced by~\eqref{eq:weaker-int-limit}.
\end{rem}

\section{Graphs of functions and the slope of the tangent plane to a graph}
\label{sec:slope}

\begin{rem}
    \label{rem:pqF-setup}
    A~convenient way to work with graphs of functions defined on some $T \in
    \grass nm$ and with values in~$T^\perp$ is to express the function using
    orthonormal bases for~$T$ and~$T^\perp$. To do that we choose orthogonal
    projections $\pp \in \OP(n,m)$ and $\qq \in \OP(n,n-m)$
    (cf.~\cite[1.7.4]{Fed69}) such that $\im \pp^\ast = T$ and $\im \qq^\ast =
    T^\perp$. Then if $A \subseteq \R^m$ and $f : A \to \R^{n-m}$ we define $F =
    \pp^\ast + \qq^\ast \circ f$ and then $\im F$ is the graph of~$f$
    with~$\R^m$ identified with~$T$.
\end{rem}

The following remark, made in the spirit of~\cite[8.9(5)]{All72}, allows to
express the ``slope'' of the tangent plane to a graph by the norm of the
derivative of the function; see~\ref{cor:tilt-for-graphs}.
\begin{rem}
    \label{rem:tilt-vs-norm}
    Assume $T \in \grass nm$ and $\eta \in \Hom(T,T^\perp)$. Set $S = \{ v +
    \eta(v) : v \in T \} \in \grass nm$. Observe that the function $[0,\infty)
    \ni t \mapsto t^2 (1 + t^2)^{-1}$ is increasing; hence,
    using~\cite[8.9(3)]{All72},
    \begin{align*}
        \| \proj S - \proj T \|^2
        = \| \pproj T \circ \proj S \|^2
        &= \sup \left\{ |\pproj T u|^2 |u|^{-2} : u \in S \without\{0\} \right\}
        \\
        &= \sup \left\{ |\eta(w)|^2 |w + \eta(w)|^{-2} : w \in T \without\{0\} \right\}
        \\
        &= \sup \left\{ |\eta(w)|^2 (1 + |\eta(w)|^2)^{-1} : w \in T \,,\, |w| = 1 \right\}
        = \frac{\|\eta\|^2}{1 + \|\eta\|^2} \,.
    \end{align*}
\end{rem}

\begin{cor}
    \label{cor:tilt-for-graphs}
    Let $\pp$, $\qq$, $T$, $A$, $f$, and $F$ be as
    in~\ref{rem:pqF-setup}. Assume that $\Sigma \subseteq \im F$, $a \in \Sigma$
    is such that $\Tan{\Sigma}{a} \in \grass nm$ and that $f$ is differentiable
    at $x = \pp(a)$. Then, employing~\ref{rem:tilt-vs-norm} with $Df(x)$ in
    place of~$\eta$, we obtain
    \begin{gather}
        \label{eq:tilt-Df}
        \| \proj{\Tan{\Sigma}{a}} - \proj T \|^2
        =
        \frac{\|Df(x)\|^2}{ 1 + \|Df(x)\|^2 }
    \end{gather}
    
    Let $b \in \Sigma$ and set $y = \pp(b)$. Then $b = F(y)$ and
    $DF(x)(y-x) \in \Tan{\Sigma}{a}$. Define
    \begin{align*}
        u &= \qq^\ast(f(y) - f(x) - Df(x)(y-x)) = F(x) - F(y) - DF(x)(y-x) \in T^\perp
        \\
        \text{and} \quad
        v &= \pproj{\Tan{\Sigma}{a}} (b-a) = \pproj{\Tan{\Sigma}{a}} u \,.
    \end{align*}
    Then, by~\eqref{eq:tilt-Df}, we get
    \begin{gather*}
        |u-v| = |\proj{\Tan{\Sigma}{a}} u| 
        = |\proj{\Tan{\Sigma}{a}} \pproj T  u|
        \le \|\proj{\Tan{\Sigma}{a}} - \proj T\| |u|
        = \frac{\|Df(x)\| \, |u|}{(1+\|Df(x)\|^2)^{1/2}}  \,.
    \end{gather*}
    In consequence,
    \begin{gather}
        \label{eq:Taylor-dist}
        \begin{aligned}
            |\pproj{\Tan{\Sigma}{a}}(b-a)|
            &\le |f(y) - f(x) - Df(x)(y-x)|  \,,
            \\
            |\pproj{\Tan{\Sigma}{a}}(b-a)|
            &\ge \left(1 - \frac{\|Df(x)\|}{(1 + \|Df(x)\|^2)^{1/2}} \right) |f(y) - f(x) - Df(x)(y-x)| \,.
        \end{aligned}
    \end{gather}
\end{cor}

\section{Higher order rectifiability criterion for graphs}
\label{sec:criterion}

To talk about approximate features of functions (e.g. limits, continuity,
differentiability; cf.~\cite[2.9.2\,,\,3.1.2\,,\,3.2.16]{Fed69}) one needs to
provide two parameters: a~measure and a~Vitali relation
(cf.~\cite[2.8.16]{Fed69}). It will be convenient to define a~standard family of
Vitali relations.
\begin{defin}
    \label{def:vitali-rel}
    For $k \in \N \without\{0\}$, we set
    \begin{gather*}
        \mathcal V_k = \bigl\{ (x,\cball xr) : x \in \R^k \,,\, r \in (0,\infty) \bigr\}  \,.
    \end{gather*}
\end{defin}
\begin{rem}
    \label{rem:vitali-rel}
    If $k \in \N \without\{0\}$ and $\phi$ is a measure over $\R^k$ such that
    all open sets are $\phi$~measurable and $\phi(A) < \infty$ for all bounded
    sets $A \subseteq \R^k$, then due to~\cite[2.8.18]{Fed69} the family
    $\mathcal V_k$ is a~$\phi$~Vitali relation.
\end{rem}

In the following proposition, whenever we write ``$\ap Df$'' we mean the
approximate differential with respect to~$(\LM^m, \mathcal V_m)$.
\begin{prop}
    \label{prop:schatzle}
    Suppose $\alpha \in (0,1]$, and $A \subseteq \R^m$ is $\LM^m$-measurable and
    such that $\density^m(\LM^m \restrict A, a) = 1$ for all $a \in A$. Let $f :
    A \to \R^{n-m}$ be $(\LM^m, \mathcal V_m)$ approximately differentiable
    on~$A$ and satisfy one of the following conditions
    \begin{gather}
        \label{eq:sch-1}
        \ \ \limsup_{r \downarrow 0} r^{-m} \int_{A \cap \cball y{r}}
        \frac{|f(z) - f(y) - \ap Df(y)(z-y)|}{|z-y|^{1+\alpha}}\ud \LM^m(z) < \infty 
        \quad \text{for all $y \in A$} \,,
    \end{gather}
    or
    \begin{gather}
        \label{eq:sch-2}
        (\LM^m, \mathcal V_m) \aplimsup_{z \to y}
        \frac{|f(z) - f(y) - \ap Df(y)(z-y)|}{|z-y|^{1+\alpha}} < \infty 
        \quad \text{for all $y \in A$} \,.
    \end{gather}
    Then there exist functions $f_k \in \cnt^{1,\alpha}(\R^m,\R^{n-m})$, such
    that
    \begin{gather}
        \label{eq:sch-rect}
        \LM^m\left(
            A \without \tcup_{k=1}^{\infty}
            \bigl\{
              x \in A : f(x) = f_k(x) \text{ and } \ap Df(x) = Df_k(x) 
            \bigr\}
        \right) = 0.
    \end{gather}
    In particular, the graph of~$f$ is countably $(\HM^m,m)$~rectifiable of
    class~$\cnt^{1,\alpha}$.
\end{prop}

\begin{proof}
    The proof can be found in~\cite[Lemma A.1]{Sch09} for the case $\alpha = 1$.
    If $0 < \alpha < 1$, exactly the same proof, with relevant occurrences
    of~$2$ replaced by $1 + \alpha$, establishes the assertion. 
\end{proof}

\begin{rem}
    \label{rem:sch-ll}
    In case condition~\eqref{eq:sch-2} is satisfied in~\ref{prop:schatzle},
    the~conclusion of~\ref{prop:schatzle} follows also from a theorem of Lin and
    Liu~\cite[Theorem~1.5]{LL13}. However, one should note that at a~single
    point $y \in A$ condition~\eqref{eq:sch-1} does \emph{not}
    imply~\eqref{eq:sch-2} at~$y$. It is rather a consequence of the proposition
    that condition~\eqref{eq:sch-1} at all points of~$A$ implies
    condition~\eqref{eq:sch-2} for $\LM^m$ almost all~$y \in A$. Inspecting the
    proof of~\cite[Lemma A.1]{Sch09} one can extract a sufficient condition,
    which is implied by~\eqref{eq:sch-1} as well as by~\eqref{eq:sch-2}, for the
    proposition to hold; namely, it is enough to assume that for all $y \in A$
    there exists some $K \in (0,\infty)$ such that
    \begin{gather*}
        \limsup_{r \downarrow 0} 
        \frac{\LM^m\bigl(\cball yr \without \bigl\{ z \in A :
          |f(z) - f(y) - \ap Df(y)(z-y)| < K r^{1 + \alpha} \bigr\}\bigr)}
        {\unitmeasure{m} r^m}
        < \varepsilon_0(m) \,,
    \end{gather*}
    where $\varepsilon_0(m) \in (0,1)$ is a small constant depending only on~$m$.
\end{rem}

\begin{cor}
    \label{cor:main-rect-cond}
    Let $\pp$, $\qq$, $T$, $A$, $f$, $F$ be as in~\ref{rem:pqF-setup}. Suppose
    $\alpha \in (0,1]$, and $A$ is $\LM^m$~measurable, and $f$ is $(\LM^m,
    \mathcal V_m)$~approximately differentiable on~$A$, and $\Sigma = F \lIm A
    \rIm$ satisfies $\HM^m(\Sigma) < \infty$. Assume that one of the following
    conditions is satisfied for $\HM^m$~almost all $a \in \Sigma$
    \begin{gather}
        \label{eq:limsup-int-cond}
        \limsup_{r \downarrow 0} r^{-m} \int_{\Sigma \cap \cball ar} 
        \frac{|\pproj{\apTan{m}{\HM^m \restrict \Sigma}a}(b-a)|}{|b-a|^{1+\alpha}} \ud \HM^m(b)  < \infty
    \end{gather}
    or
    \begin{gather}
        \label{eq:aplimsup-cond}
        (\HM^m \restrict \Sigma,\mathcal V_n) \aplimsup_{b \to a}
        \frac{|\pproj{\apTan{m}{\HM^m \restrict \Sigma}a}(b-a)|}{|b-a|^{1+\alpha}} < \infty \,.
    \end{gather}
    Then $\Sigma$ is $\HM^m$~measurable and countably $(\HM^m,m)$~rectifiable of
    class~$\cnt^{1,\alpha}$.
\end{cor}

\begin{proof}
    Employing~\cite[3.1.8]{Fed69} we can divide $A$ into a countable family of
    $\LM^m$~measurable sets $\{ A_i : i \in \N\}$ such that $f$ restricted to
    each of $A_i$ is Lipschitz and $\bigcup_{i \in \N} A_i = A$. Then $F|_{A_i}$
    is bilipschitz and, since~$\HM^m$ and~$\LM^m$ are Borel regular, $\Sigma_i =
    F \lIm A_i \rIm$ is $\HM^m$~measurable for each $i \in \N$. Hence, $\Sigma =
    \bigcup_{i \in \N} \Sigma_i$ is also $\HM^m$~measurable. Moreover, if one of
    the conditions~\eqref{eq:limsup-int-cond} or~\eqref{eq:aplimsup-cond} is
    satisfied for $\HM^m$~almost all $a \in \Sigma$, then the same condition
    holds for $\HM^m$~almost all $a \in \Sigma_i$ for each $i \in \N$. Hence, it
    suffices to prove the Corollary separately for each $A_i$ and~$\Sigma_i$ in
    place of~$A$ and~$\Sigma$. In the sequel we will assume this replacement has
    been done and that $f$ has been extended to the whole of~$\R^m$ by means of
    the Kirszbraun's theorem~\cite[2.10.43]{Fed69}, so that we have
    \begin{gather*}
        f : \R^m \to \R^{n-m}
        \quad \text{satisfies} \quad
        L = \lip(f) < \infty 
        \\
        \text{and} \quad
        \apTan{m}{\HM^m \restrict \Sigma}a = \Tan{\Sigma}a
        \quad \text{for $\HM^m$ almost all $a \in \Sigma$.}
    \end{gather*}

    If~\eqref{eq:limsup-int-cond} holds for $\HM^m$~almost all $a \in \Sigma$,
    then let $\Sigma' \subseteq \Sigma$ be the set of $a \in \Sigma$ for
    which~\eqref{eq:limsup-int-cond} holds. If~\eqref{eq:aplimsup-cond} holds
    for $\HM^m$~almost all $a \in \Sigma$, let $\Sigma' \subseteq \Sigma$ be the
    set of $a \in \Sigma$ for which~\eqref{eq:aplimsup-cond} holds.
    Since $\HM^m(\Sigma \without \Sigma') = 0$, we know $\Sigma'$ is
    $\HM^m$~measurable. Recall the definitions of $\pp$, $\qq$, and $F$
    from~\ref{rem:pqF-setup}. Set $B' = \pp \lIm \Sigma' \rIm$ and note that
    $B' = F^{-1} \lIm \Sigma' \rIm$ so it is $\LM^m$~measurable. Next, set
    $\tilde B = \{ x \in B' : Df(x) \text{ exists} \}$. Then $\LM^m(B' \without
    \tilde B) = 0$ due to the Rademacher's theorem (cf.~\cite[3.1.6]{Fed69});
    hence, $\tilde B$ is also $\LM^m$~measurable. Define $B = \{ x \in \tilde B
    : \density^m(\LM^m \restrict \tilde B, x) = 1\}$. Then,
    by~\cite[2.9.11]{Fed69}, $B$~is $\LM^m$~measurable, $\LM^m(\tilde B \without
    B) = 0$ and $\density^m(\LM^m \restrict B,x) = 1$ for \emph{all} $x \in
    B$. Observe
    \begin{equation}
        \label{eq:Sigma-FB}
        \HM^m(\Sigma \without F \lIm B \rIm)
        = \HM^m(\Sigma \without \Sigma') + \HM^m(F \lIm B' \without B \rIm) = 0 
        \quad \text{because $F$ is Lipschitz;}
    \end{equation}
    hence, it suffices to check that~\ref{prop:schatzle} applies to~$f|_B$.

    Set $\lambda = (1+L^2)^{-1/2} \in (0,1]$ and note that $\lip(F) \le
    \lambda^{-1}$; hence,
    \begin{equation}
        \label{eq:FBlr-Br}
        F \lIm \cball{\pp(a)}{\lambda r} \cap B \rIm 
        \subseteq \cball ar \cap F \lIm B \rIm
        \quad \text{for each $a \in \Sigma$ and $r \in (0,\infty)$;}
    \end{equation}
    Employing~\eqref{eq:Sigma-FB} combined with~\eqref{eq:Taylor-dist} and then
    applying the area formula~\cite[3.2.3]{Fed69} together
    with~\eqref{eq:FBlr-Br}, we obtain
    \begin{multline*}
        r^{-m} \int_{\cball ar \cap \Sigma}
        \frac{|\pproj{\Tan{\Sigma}a}(b-a)|}{|b-a|^{1+\alpha}} \ud \HM^m(b)
        \\
        \ge 
        \lambda^{1+\alpha}(1 - \lambda L) r^{-m} \int_{\cball ar \cap F \lIm B \rIm}
        \frac{|f(\pp(b)) - f(\pp(a)) - Df(\pp(a))(\pp(b)
          - \pp(a))|}{|\pp(b) - \pp(a)|^{1+\alpha}} \ud \HM^m(b)
        \\
        \ge 
        \lambda^{1+\alpha+m}(1 - \lambda L) (\lambda r)^{-m} \int_{\cball{x}{\lambda r} \cap B}
        \frac{|f(y) - f(x) - Df(x)(y - x)|}{|y - x|^{1+\alpha}} \ud \LM^m(y)
    \end{multline*}
    for $r \in (0,\infty)$, $x \in B$, and $a = F(x)$. Hence,
    if~\eqref{eq:limsup-int-cond} holds, then one can employ~\ref{prop:schatzle}
    to see that $F \lIm B \rIm$ is countably $(\HM^m,m)$~rectifiable of
    class~$\cnt^{1,\alpha}$ and, due to~\eqref{eq:Sigma-FB}, so is~$\Sigma$.

    Fix $a \in F \lIm B \rIm$ and set $x = \pp(a)$. For $y \in B$ and $b \in F
    \lIm B \rIm$ define
    \begin{gather*}
        g(b) = \frac{|\pproj{\Tan{\Sigma}a}(b-a)|}{|b-a|^{1+\alpha}}  \,,
        \quad
        h(y) = \frac{|f(y) - f(x) - Df(x)(y - x)|}{|y - x|^{1+\alpha}} \,,
        \\
        \text{and} \quad
        \phi = \HM^m \restrict F \lIm B \rIm = \HM^m \restrict \Sigma \,,
    \end{gather*}
    Setting $\Delta = \lambda^{1+\alpha}(1 - \lambda L)$ we obtain,
    by~\eqref{eq:Taylor-dist} and the area formula~\cite[3.2.3]{Fed69},
    \begin{gather*}
        \Delta h(\pp(b)) \le g(b) 
        \quad \text{and} \quad
        \LM^m(S) \le \phi(F \lIm S \rIm) \le \lambda^{-m} \LM^m(S)
    \end{gather*}
    whenever $b \in F \lIm B \rIm$ and $S \subseteq B$ is
    $\LM^m$~measurable. Hence, for each $r,t \in (0,\infty)$
    \begin{gather*}
        \{ y \in B : h(y) > t \} \subseteq \pp \bigl\lIm \{ b \in F \lIm B \rIm : g(b) > \Delta t \} \bigr\rIm \,,
        \\
        \frac{\LM^m(\cball x{\lambda r} \cap \{ y \in B : h(y) > t \})}{\LM^m(\cball xr \cap B)}
        \le \frac{\phi(\cball ar \cap \{ b \in F \lIm B \rIm : g(b) > \Delta t \})}{\lambda^m \phi(\cball ar)} \,.
    \end{gather*}
    Therefore,
    \begin{multline}
        \label{eq:inf-cmp}
        \inf \left\{ t \in \R : \lim_{r \downarrow 0}
            \frac{\LM^m(\cball x{\lambda r} \cap \{ y \in B : h(y) > t \})}{\LM^m(\cball xr \cap B)} = 0
        \right\}
        \\
        \le 
        \inf \left\{ t \in \R : \lim_{r \downarrow 0}
            \frac{\phi(\cball ar \cap \{ b \in F \lIm B \rIm : g(b) > \Delta t \})}{\lambda^m \phi(\cball ar)} = 0
        \right\} \,.
    \end{multline}
    For any $x \in B$ we have $\density^m(\LM^m \restrict B,x) = 1$ so it follows that
    \begin{gather}
        \label{eq:psi-doubling}
        \lim_{r \downarrow 0} \frac{\LM^m(\cball x{\lambda r})}{\LM^m(\cball xr \cap B)} = \lambda^{m} < \infty \,.
    \end{gather}
    Recalling $x = \pp(a) \in B$ was chosen arbitrarily and
    combining~\eqref{eq:inf-cmp} with~\eqref{eq:psi-doubling} yields
    \begin{gather*}
        (\LM^m, \mathcal V_m)\aplimsup_{y \to x} h(y) \le (\phi, \mathcal V_n)\aplimsup_{b \to a} g(b) 
    \end{gather*}
    for all $x \in B$ and $a = F(x)$. Consequently, if~\eqref{eq:aplimsup-cond}
    holds, then one can employ~\ref{prop:schatzle} to see that $F \lIm B \rIm$
    is countably $(\HM^m,m)$~rectifiable of class~$\cnt^{1,\alpha}$ and, because
    of~\eqref{eq:Sigma-FB}, so is~$\Sigma$.
\end{proof}

\section{Existence of balanced balls}
\label{sec:balanced}

\begin{defin}
    \label{def:X-delta}
    For $\delta \in [0,1]$, and $a \in \R^n$, and $r \in (0,\infty]$ we set
    \begin{displaymath}
        X_{\delta}(a,r) = \bigl\{
        (b_1,\ldots,b_m) \in {\cball ar}^m : |(b_1-a) \wedge \cdots \wedge (b_m-a)| \ge \delta r^m
        \bigr\} \,.
    \end{displaymath}
\end{defin}

The following lemma~\ref{lem:balanced-balls} is a variant
of~\cite[Lemma~3.1]{AT15}. Similar results can also be found
in~\cite[Proposition~3.1]{LW09}, and \cite[Lemma~5.8]{DS91},
and~\cite[Lemma~4.2]{Meu15a}.

\begin{lem}
    \label{lem:balanced-balls}
    Suppose
    \begin{gather*}
        \text{$\mu$ is a Radon measure over~$\R^n$}  \,,
        \quad
        a,b \in \R^n \,,
        \quad
        r \in (0,\infty) \,,
        \quad
        \mu(\cball ar) > 0 \,,
        \\
        t,\gamma \in (0,1) \,,
        \quad
        k \in \N \,,
        \quad
        k < m \,,
        \quad
        L_k \in \grass nk \,.
    \end{gather*}
    Then one of the following alternatives holds:
    \begin{enumerate}
    \item
        \label{i:bb:balanced}
        There exist $x_{k+1}, \ldots, x_{m} \in \cball ar$ such that if $L_j =
        L_k + \lin\{ x_{k+1}-b, \ldots, x_j-b \}$ for $j = k+1, \ldots, m$, then
        \begin{gather*}
            \dist(x_j-b,L_{j-1}) > \gamma r 
            \quad \text{for $j = k+1,\ldots,m$} \,,
            \\
            \mu(\cball{x_j}{tr} \cap \cball ar) \ge \Gamma^{-1} t^n \mu(\cball ar) 
            \quad \text{for $j = k+1,\ldots,m$} \,,
        \end{gather*}
        where $\Gamma = \Gamma(n) \in [1,\infty)$.
    \item
        \label{i:bb:unbalanced}
        There exist $\lambda,N \in \N$ and $L_\lambda \in \grass n{\lambda}$
        and $y_1, \ldots, y_N \in \cball ar \cap (b + L_\lambda)$ satisfying
        \begin{gather*}
            k \le \lambda < m \,,
            \quad
            N \le \Gamma \gamma^{-\lambda} \,,
            \quad
            L_k \subseteq L_\lambda \,,
            \\
            \bigl\{ \cball{y_i}{40\gamma r} : i = 1,\ldots,N \bigr\} \text{ is disjointed} \,,
            \\
            \sum_{i = 1}^N \mu(\cball{y_i}{4 \gamma r}) \ge \Gamma^{-1} \mu(\cball ar) \,,
            \\
            \frac{\mu(\cball{y_i}{4 \gamma r})}{(4 \gamma r)^m}
            \ge \Gamma^{-1} \gamma^{-(m-\lambda)} \frac{\mu(\cball ar)}{r^m} 
            \quad \text{for $i = 1,\ldots,N$} \,,
        \end{gather*}
        where $\Gamma = \Gamma(m) \in [1,\infty)$.
    \end{enumerate}
\end{lem}

\begin{proof}
    We~mimic the proof of~\cite[Lemma~3.1]{AT15}.

    Set $\varepsilon = 2^{-(n+1)} t^n$. Choose $x_{k+1}$, \ldots, $x_m$
    inductively so that for $j = k+1, \ldots, m$
    \begin{gather*}
        L_j = L_k + \lin \bigl\{ x_{x+1} - b, \ldots, x_{j} - b \bigr\} \,,
        \\
        x_j \in \cball ar \without \bigl( L_{j-1} + \cball 0{\gamma r} \bigr) \,,
        \\
        s_j = \sup \bigl\{ \mu(\cball x{tr} \cap \cball ar) 
        : x \in \cball ar \without \bigl( L_{j-1} + \cball 0{\gamma r} \bigr) \bigr\} \,.
        \\
        \mu(\cball{x_j}{tr} \cap \cball ar) \ge \tfrac 12 s_j \,.
    \end{gather*}
    If $s_j \ge \varepsilon \mu(\cball ar)$ for $j = k+1,\ldots,m$, then
    alternative \ref{i:bb:balanced} holds with $\Gamma = 2^{n+1}$.

    Assume there exists $\lambda \in \{k,\ldots,m-1\}$ such that $s_{\lambda+1}
    \le \varepsilon \mu(\cball ar)$. Then, since $\cball ar \without \bigl(
    L_{\lambda} + \cball 0{\gamma r} \bigr)$ can be covered by $2^n t^{-n}$ balls
    with centers in $\cball ar \without \bigl( L_{\lambda} + \cball 0{\gamma r}
    \bigr)$ and radii~$tr$, we get
    \begin{displaymath}
        \mu\bigl( \cball ar \without \bigl( L_{\lambda} + \cball 0{\gamma r} \bigr) \bigr)
        \le 2^{n} t^{-n} \varepsilon \mu(\cball ar) \,.
    \end{displaymath}
    Recalling $\varepsilon = \frac 12 2^{-n} t^{n}$, this implies that
    \begin{equation}
        \label{eq:bb:strip-meas}
        \mu\bigl( \cball ar \cap \bigl( L_{\lambda} + \cball 0{\gamma r} \bigr) \bigr)
        \ge \tfrac 12 \mu(\cball ar) \,.
    \end{equation}
    Since $\dim(L_{\lambda}) = \lambda$ there exists a finite set $I \subseteq \cball ar
    \cap L_{\lambda}$ such that
    \begin{gather}
        \cball ar \cap \bigl( L_{\lambda} + \cball 0{\gamma r} \bigr)
        \subseteq \tbcup \bigl\{ \cball{z}{4 \gamma r} : z \in I \bigr\} \,,
        \\ 
        \label{eq:bb:half-disjointed}
        \bigl\{ \cball{z}{\tfrac 12 \gamma r} : z \in I \bigr\} \text{ is disjointed}  \,,
        \quad
        \HM^0(I) \le K = (4\gamma)^{-\lambda} \,.
    \end{gather}
    Next, we define $J = \bigl\{ z \in I : \mu(\cball{z}{4 \gamma r}) \ge
    (4K)^{-1} \mu(\cball ar) \bigr\}$ and note that
    \begin{equation}
        \mu\bigl( \tbcup \bigl\{ \cball{z}{4 \gamma r} : z \in I \without J \bigr\} \bigr)
        \le \tfrac 14 \mu(\cball ar) \,;
    \end{equation}
    hence, employing~\eqref{eq:bb:strip-meas},
    \begin{equation}
        \label{eq:bb:cover-fourth}
        \mu\bigl( \tbcup \bigl\{ \cball{z}{4 \gamma r} : z \in J \bigr\} \bigr)
        \ge \tfrac 14 \mu(\cball ar) \,.
    \end{equation}
    Now we construct $Y = \{ y_1,\ldots, y_N \} \subseteq J$ inductively so that
    for $i = 1,\ldots,N$
    \begin{gather}
        J_1 = J \,,
        \quad
        J_{i} = \bigl\{z \in J
        : \cball{z}{40 \gamma r} \cap \cball{y_j}{40 \gamma r} = \varnothing
          \text{ for $j = 1,2,\ldots,i-1$} \bigr\} 
          \quad \text{if $i \ge 2$}\,,
        \\
        \label{eq:bb:biggest-choice}
        y_i \in J_i \,,
        \quad
        \mu(\cball{y_i}{4 \gamma r}) \ge \mu(\cball{z}{4 \gamma r})
        \quad \text{for } z \in J_{i} \,,
        \quad J_{N+1} = \varnothing \,.
    \end{gather}
    For $i = 1,2,\ldots,N$ we see from~\eqref{eq:bb:half-disjointed} that
    $J_{i} \without J_{i+1}$ contains at most $20^{\lambda}$ points. Thus,
    using~\eqref{eq:bb:cover-fourth} and~\eqref{eq:bb:biggest-choice}, we obtain
    \begin{displaymath}
        \mu(\cball ar)
        \le 4 \sum_{i = 1}^N \mu \bigl( \tbcup \bigl\{ \cball{z}{4 \gamma r} 
        : z \in J_i \without J_{i+1} \bigr\} \bigr)
        \le 4 \cdot 20^{\lambda} \sum_{i = 1}^N \mu \bigl( \cball{y_{i}}{4 \gamma r} \bigr) \,.
    \end{displaymath}
    Finally, since $\mu(\cball{z}{4 \gamma r}) \ge (4K)^{-1} \mu(\cball ar)$ for
    $z \in J$ and $K = (4\gamma)^{-\lambda}$, we conclude that
    \begin{displaymath}
        \frac{\mu(\cball{z}{4 \gamma r})}{(4 \gamma r)^m}
        \ge \frac{\mu(\cball ar)}{4K (4 \gamma r)^m}
        \ge 4^{-2m} \gamma^{-(m-j+1)} \frac{\mu(\cball ar)}{r^m}
        \quad \text{for $z \in J$} \,.
    \end{displaymath}
    Hence, alternative \ref{i:bb:unbalanced} holds with $\Gamma = 4 \cdot 20^m$.
\end{proof}

\begin{cor}
    \label{cor:fat-simp-exist}
    Suppose $a \in \R^n$, and $r_0 \in (0,\infty)$, and $A \in [1,\infty)$,
    \begin{gather*}
        \text{and} \quad
        A^{-1} \unitmeasure{m} r^m \le \mu(\cball ar)
        \quad \text{for $r \in (0,r_0]$} \,,
        \\
        \text{and} \quad
        \mu(\cball zr) \le A \unitmeasure{m} r^m
        \quad \text{for $z \in \R^n$ and $r \in (0,r_0]$} \,.
    \end{gather*}
    Then there exist $\delta = \delta(A,m) \in (0,1]$ and $\sigma =
    \sigma(A,n,m) \in (0,1]$ such that
    \begin{gather*}
        \mu^m(X_{\delta}(a,r)) \ge \sigma \mu(\cball ar)^m 
        \quad \text{for $r \in (0,4r_0]$}\,.
    \end{gather*}
\end{cor}

\begin{proof}
    Let $r \in (0,r_0]$ and set $\gamma = \tfrac 12
    (\Gamma_{\ref{lem:balanced-balls}\ref{i:bb:unbalanced}} A^2)^{1/(m-1)}$ and
    $t = \bigl( 1 + \frac 12 \gamma^m \bigr)^{1/m} - 1$ and $k = 0$ and $L_k =
    \{0\}$ and $b = a$ and $\delta = \frac 12 \gamma^m$ and
    apply~\ref{lem:balanced-balls} with this choice of $a$, $b$, $\gamma$, $r$,
    $t$, $k$, $L_k$. Observe that for this particular~$\gamma$ alternative
    \ref{lem:balanced-balls}\ref{i:bb:unbalanced} cannot hold due to the assumed
    lower bound on $\mu(\cball ar)$ and upper bound on $\mu(\cball z{4 \gamma
      r})$ for $z \in \R^n$. Therefore, there exist $x_1, \ldots, x_m \in \cball
    ar$ such that alternative~\ref{lem:balanced-balls}\ref{i:bb:balanced}
    holds. Then clearly
    \begin{displaymath}
        | (x_1 - a) \wedge \cdots \wedge (x_m - a) | \ge \gamma^m r^m \,.
    \end{displaymath}
    A simple computation as in~\cite[Proposition~1.7]{Kol15a} shows that if we
    choose arbitrary points $y_i \in \cball{x_i}{tr}$ for $i = 1,\ldots,m$, then
    \begin{displaymath}
        | (y_1 - a) \wedge \cdots \wedge (y_m - a) | \ge \tfrac 12 \gamma^m r^m = \delta r^m \,.
    \end{displaymath}
    This shows that
    \begin{displaymath}
        \bigl(\cball{x_1}{tr} \cap \cball ar \bigr)
        \times \cdots \times
        \bigl(\cball{x_m}{tr} \cap \cball ar \bigr) 
        \subseteq X_{\delta}(a,r) \,;
    \end{displaymath}
    hence,~\ref{lem:balanced-balls}\ref{i:bb:balanced} yields
    \begin{displaymath}
        \mu^m(X_{\delta}(a,r))
        \ge \Gamma_{\ref{lem:balanced-balls}\ref{i:bb:balanced}}^{-m} t^{nm} \mu(\cball ar)^m \,.
        \qedhere
    \end{displaymath}
\end{proof}

\section{Proof of higher order rectifiability}
\label{sec:hor-menger}

\begin{rem}
    \label{rem:mlc-lower-bound}
    If $\delta \in [0,1]$, and $a,c \in \R^n$, and $r \in (0,\infty]$, and
    $(b_1,\ldots,b_m) \in X_\delta(a,r)$, and $P = \lin\{b_1 - a, \ldots, b_m -
    a\}$, then
    \begin{gather*}
        \mlc(a,b_1,\ldots,b_m,c) \ge \frac{\delta \dist(c-a,P)}{2^m(m+1)!\,2r} \,.
    \end{gather*}
\end{rem}


\begin{lem}
    \label{lem:planes}
    Let $r,\delta \in (0, \infty)$, $\varepsilon \in (0,1)$, $P,Q \in \grass nm$, $v_1,\ldots,v_m \in \R^n$
    satisfy
    \begin{gather*}
        Q = \lin\{v_1, \ldots, v_m\} \,,
        \quad
        |v_1 \wedge \cdots \wedge v_m| \ge \delta r^m  \,,
        \quad
        |v_i| \le r \,,
        \quad
        |\pproj P v_i| \le \varepsilon r
    \end{gather*}
    for $i = 1,\ldots,m$. Then $\|\proj P - \proj Q\| \le m \delta^{-1}
    \varepsilon$.
\end{lem}

\begin{proof}
    By~\cite[8.9(3)]{All72}, there exists $u \in Q$ such that
    \begin{gather*}
        |u| = 1
        \quad \text{and} \quad
        \| \proj P - \proj Q \| = \| \pproj P \circ \proj Q \| = |\pproj P u| \,.
    \end{gather*}
    Choose $\alpha_1,\ldots,\alpha_m \in \R$ such that $u = \sum_{i=1}^m
    \alpha_i v_i$. For each $i = 1,\ldots,m$ we have
    \begin{gather*}
        \alpha_i = \bigl( v_1 \wedge \cdots \wedge v_{i-1} \wedge u \wedge v_{i+1} \wedge \cdots \wedge v_m \bigr)
        \bullet \frac{v_1 \wedge \cdots \wedge v_m}{|v_1 \wedge \cdots \wedge v_m|^2} 
        \\
        \text{and} \quad
        |\alpha_i| = \frac
        {|v_1 \wedge \cdots \wedge v_{i-1} \wedge u \wedge v_{i+1} \wedge \cdots \wedge v_m|}
        {|v_1 \wedge \cdots  \wedge v_m|}
        \le \frac 1{\delta r} \,;
        \\
        \text{hence,} \quad
        \| \proj P - \proj Q \| 
        = | \pproj P u |
        \le \sum_{i=1}^m |\alpha_i| |\pproj P v_i|
        \le m \delta^{-1} \varepsilon \,.
        \qedhere
    \end{gather*}
\end{proof}


\begin{rem}
    \label{rem:czybyszew}
    We shall frequently use the Chebyshev's inequality in the following
    form. Whenever $\nu$ measures some set $X$, $f : X \to \R$ is a
    $\nu$~measurable function, $t \in (0,\infty)$, and $A \subseteq X$ is
    $\nu$~measurable, then for any $K \in (0,\infty)$, we have
    \begin{gather*}
        \nu\left(\left\{ x \in A : |f(x)| >  K \tfint{A}{} |f| \ud \nu\right\}\right) \le K^{-1} \nu(A) \,.
    \end{gather*}
\end{rem}

\begin{lem}
    \label{lem:height-control}
    Assume $l \in \{1,2,\ldots,m+2\}$, and $\alpha \in (0,1]$, and $p \in
    [1,\infty)$, and $\mu$ is a Radon measure, and $\delta, \sigma \in (0,1)$,
    and $A \in [1,\infty)$, and $r_0 \in (0,\infty)$. Define $S$ to be the set
    of those $a \in \R^n$ for which
    \begin{gather}
        \label{eq:hc:k-fin}
        \kav^{l,p,\alpha}_{\mu}(a,4r_0) < \infty \,,
        \\
        \label{eq:hc:ad-reg}
        A^{-1} \unitmeasure{m} r^m \le \mu(\cball ar) \le A \unitmeasure{m} r^m 
        \quad \text{for $r \in (0,r_0]$} \,,
        \\
        \label{eq:hc:fat-simp}
        \mu^m(X_{\delta}(a,r)) \ge \sigma \mu({\cball ar})^m
        \quad \text{$r \in (0,r_0]$} \,.
    \end{gather}
    Then there exists a constant $C = C(m,l,p,\sigma,\alpha,\delta,A)$ and for
    each $a \in S$ there exists $T(a) \in \grass nm$ such that
    \begin{enumerate}
    \item in case $l < m+2$: for $\mu$ almost all $b \in \cball a{r_0}$
        \begin{align*}
            |\pproj{T(a)} (b-a)| &\le C \kav^{l,p,\alpha}_{\mu}(a,|b-a|)^{1/p} |b-a|^{1 + \alpha} 
            \intertext{and, whenever $b \in S \cap \cball{a}{\frac 12 r_0}$,}
            \|\proj{T(a)} - \proj{T(b)} \| &\le C \kav^{l,p,\alpha}_{\mu}(a,|b-a|)^{1/p} |b-a|^{\alpha}  \,;
        \end{align*}
    \item in case $l = m+2$: for any $r \in (0,r_0]$
        \begin{displaymath}
            \biggl( \fint_{\cball ar} \dist(c-a, T(a))^p \ud \mu(c) \biggr)^{1/p}
            \le C \kav^{l,p,\alpha}_{\mu}(a,4r)^{1/p} r^{1+\alpha} \,;
        \end{displaymath}
    \end{enumerate}
    In particular $\apTan{m}{\mu}{a} \subseteq T(a)$ for all $a \in S$,
    by~\ref{prop:apTan-subset-T} and~\ref{rem:r-limit}.
\end{lem}

\begin{proof}
    Obviously we can assume $S$ is not empty -- otherwise there is nothing to
    prove. Set $M = (2^{m + m^2} A^{2m} + 2) \sigma^{-1}$. If $2 \le l \le m+1$,
    define
    \begin{displaymath}
        Y(a,r) = \biggl\{
            (b_1,\ldots,b_m) \in {\cball ar}^m 
            : \mlc^{l,p,\alpha}_{\mu,a,r}(b_1,\ldots,b_{l-1})
            > \frac{M \kav^{l,p,\alpha}_{\mu}(a,r)}{ \mu(\cball ar)^{l-1} }
        \biggr\} \,,
    \end{displaymath}
    if $l=m+2$, set
    \begin{displaymath}
        Y(a,r) = \biggl\{
            (b_1,\ldots,b_m) \in {\cball ar}^m :
            \int_{\cball ar} \mlc^{l,p,\alpha}_{\mu,a,r}(b_1,\ldots,b_m,c) \ud \mu_{\Sigma}^1(c)
            > \frac{M \kav^{l,p,\alpha}_{\mu}(a,r)}{\mu(\cball ar)^{m}}
        \biggr\} \,,
    \end{displaymath}
    and if $l = 1$, set $Y(a,r) = \varnothing$. Employing Chebyshev's
    inequality~\ref{rem:czybyszew} we obtain
    \begin{equation}
        \label{eq:czybyszew}
        \mu^m(Y(a,r)) \le M^{-1} \mu({\cball ar})^m
    \end{equation}
    for $l \in \{1,\ldots,m+2\}$, and $a \in \R^n$, and $r \in (0,\infty]$.
    Since $M > \sigma^{-1}$, using~\eqref{eq:hc:fat-simp}, we get
    \begin{gather}
        \label{eq:X-Y-meas}
        \mu^m(X_{\delta}(a,r) \without Y(a,r))
        \ge \left( \sigma - \tfrac 1M \right) \mu(\cball ar)^m > 0 
        \quad \text{for $a \in S$ and $0 < r \le r_0$} \,.
    \end{gather}
    For $a \in S$, and $0 < r \le r_0$, and $(g_1,\ldots,g_m) \in
    X_{\delta}(a,r) \without Y(a,r)$ if $P = \lin\{g_1-a,\ldots,g_m-a\}$ and $1
    \le l \le m+1$, then using~\ref{rem:mlc-lower-bound}
    together with~\eqref{eq:hc:ad-reg} and~\eqref{eq:hc:fat-simp} we get
    \begin{multline*}
        \frac{M \kav^{l,p,\alpha}_{\mu}(a,r)}{\left( A^{-1} \unitmeasure{m} r^m \right)^{l-1}}
        \ge \frac{M \kav^{l,p,\alpha}_{\mu}(a,r)}{\mu(\cball ar)^{l-1}} 
        \ge \mlc^{l,p,\alpha}_{\mu,a,r}(g_1,\ldots,g_{l-1})
        \\
        \ge \esssups{\mu}_{b \in \cball ar}
        \frac{\mlc(a, g_1, \ldots, g_m, b)^p}{\diam(\{a, g_1, \ldots, g_m, b\})^{m(l-1) + \alpha p}}
        \ge 
        \frac{\esssups{\mu}_{b \in \cball ar} \bigl( \delta \dist(b-a,P)^p \bigr)}{ (2^m(m+1)!)^p (2r)^{m(l-1) + (1+\alpha) p}}
    \end{multline*}
    which implies
    \begin{gather}
        \label{eq:P-dist}
        \esssups{\mu}_{b \in \cball ar} \bigl( \dist(b-a, P) \bigr)
        \le C_1 \kav^{l,p,\alpha}_{\mu}(a,r)^{1/p} r^{1 + \alpha} \,,
        \\ \notag
        \text{where} \quad
        C_1 = A^{1/p} M^{1/p} (m+1)!\, \unitmeasure{m}^{-1/p} 2^{m+m(l-1)/p+1+\alpha} \,.
    \end{gather}
    An~analogous computation shows that in case $l=m+2$, for $a \in S$, and $0 <
    r \le r_0$, and $(g_1,\ldots,g_m) \in X_{\delta}(a,r) \without Y(a,r)$ if $P
    = \lin\{g_1-a,\ldots,g_m-a\}$, then
    \begin{gather}
        \label{eq:P-mean-dist}
        \biggl( \fint_{\cball ar} \dist(c-a, P)^p \ud \mu(c) \biggr)^{1/p}
        \le C_1 \kav^{l,p,\alpha}_{\mu}(a,r)^{1/p} r^{1 + \alpha} \,.
    \end{gather}

    Now, we shall prove the lemma in case $1 \le l \le m+1$. Set $E = \{ x \in
    \R^n : \density^{*m}(\mu,x) > 0 \}$. Due to~\eqref{eq:X-Y-meas}, for each $a
    \in S$ and $0 < r \le r_0$ there exists an~$m$-tuple
    \begin{gather*}
        \bigl( g_1(a,r),\ldots,g_m(a,r) \bigr) \in E^m \cap X_{\delta}(a,r) \without Y(a,r)
    \end{gather*}
    and we can define
    \begin{gather*}
        P(a,r) = \lin \bigl\{ (g_1(a,r) - a), \ldots, (g_m(a,r) - a) \bigr\} \in \grass nm \,.
    \end{gather*}
    Whenever $a \in S$ and $0 \le s \le r \le r_0$, noting $g_i(a,s) \in E \cap
    \cball ar$ for $i = 1,\ldots,m$, we may employ~\eqref{eq:P-dist} together
    with~\ref{lem:planes} to obtain
    \begin{gather}
        \label{eq:P-conv-rate}
        \| \proj{P(a,r)} - \proj{P(a,s)}\|
        \le m \delta^{-1} C_1  \kav^{l,p,\alpha}_{\mu}(a,r)^{1/p} r^{\alpha}  \,.
    \end{gather}
    Therefore, for each $a \in S$, the spaces $P(a,r)$ converge as $r \to 0$ to
    some $T(a) \in \grass nm$ and
    \begin{gather}
        \label{eq:P-osc}
        \| \proj{P(a,r)} - \proj{T(a)}\|
        \le C_2 \kav^{l,p,\alpha}_{\mu}(a,r)^{1/p} r^{\alpha} \,,
        \quad \text{where }
        C_2 = m \delta^{-1} C_1  \,.
    \end{gather}
    Moreover, by~\eqref{eq:P-dist} and the triangle inequality, for any $a \in
    S$ and $b \in E \cap \cball a{r_0}$
    \begin{gather*}
        |\pproj{T(a)} (b-a)| \le (C_1 + C_2) \kav^{l,p,\alpha}_{\mu}(a,|b-a|)^{1/p} |b-a|^{1 + \alpha} \,.
    \end{gather*}

    Assume $a \in S$, and $r \in (0,r_0]$, and $b \in S \without\{a\}$ are such
    that $|b-a| = \frac 12 r$. Then for each $i = 1,\ldots,m$ there holds
    $|g_i(b,\frac 12 r) - a| \le r$ and it follows from~\eqref{eq:P-dist} that
    \begin{gather*}
        \left|\proj{P(a,r)^{\perp}}\bigl(g_i(b,\tfrac 12 r) - a\bigr)\right| 
        \le 2 C_1 \kav^{l,p,\alpha}_{\mu}(a,r)^{1/p} r^{1 + \alpha} \,;
    \end{gather*}
    hence, employing again~\ref{lem:planes}, we get
    \begin{gather}
        \label{eq:P-ab-dist}
        \bigl\| \proj{P(a,r)} - \proj{P\bigl(b,\tfrac 12 r\bigr)} \bigr\| 
        \le 2^{1+\alpha} C_2 \kav^{l,p,\alpha}_{\mu}(a,r)^{1/p} |b-a|^{\alpha} \,.
    \end{gather}
    In consequence, for all $a,b \in S$, $r \in (0,\infty)$ with $|a-b| = \frac
    12 r \le \frac 12 r_0$
    \begin{align*}
        \| \proj{T(a)} - \proj{T(b)} \| 
        &\le \| \proj{T(a)} - \proj{P(a,r)}\|
        + \| \proj{P(a,r)} - \proj{P\left(b,\tfrac r2\right)}\| 
        + \| \proj{P\left(b, \tfrac r2\right)} - \proj{T(b)}\|
        \\
        &\le C_3 \kav^{l,p,\alpha}_{\mu}(a,r)^{1/p} |b-a|^{\alpha} \,,
        \qquad \text{where $C_3 = C_2(2 + 2^{1+\alpha})$.}
    \end{align*}
    This finishes the proof in case $1 \le l \le m+1$.

    Next, we shall consider the case $l=m+2$. For $a \in S$ and $i = \N$ define
    inductively
    \begin{gather*}
        \rho_{i} = 2^{-i} r_0 \,,
        \quad
        Q_{0}(a) = P(a,\rho_{0}) \,,
        \\
        Z_i(a) = \biggl\{
            c \in \Sigma(a,\rho_i) : \dist(c-a, Q_i(a))^p > M \fint_{\cball a{\rho_i}} \dist(z-a,Q_i(a))^p \ud \mu(z)
        \biggr\} \,,
        \\
        W_i(a) = \bigl\{
            (c_1,\ldots,c_m) \in \cball a{\rho_i}^m : c_j \in Z_i(a) \text{ for some } j \in \{1,\ldots,m\}
        \bigr\} \,,
        \\
        \intertext{and, whenever $i \ge 1$,}
        \bigl( h_{i,1}(a), \ldots, h_{i,m}(a) \bigr) \in X(a,\rho_i) \without \left( Y(a,\rho_i) \cup W_{i-1}(a) \right) \,,
        \\
        Q_i(a) = \lin \bigl\{ h_{i,1}(a) - a, \ldots, h_{i,m}(a) - a \bigr\} \,.
    \end{gather*}
    Note that $(h_{i,1}(a), \ldots, h_{i,m}(a))$ exists for all $i \in \N$ and
    $a \in S$. Indeed, for~$i \in \N$ and~$a \in S$ Chebyshev's
    inequality~\ref{rem:czybyszew} yields
    \begin{gather*}
        \mu(Z_i(a)) \le M^{-1} \mu(\cball a{\rho_i}) \,;
        \\
        \text{hence,} \quad
        \mu^m(W_i(a))
        \le \bigl( \bigl(1 + M^{-1}\bigr)^m - 1 \bigr) \mu(\cball a{\rho_i})^m \,,
    \end{gather*}
    which implies for $i \in \N \without\{0\}$, combining~\eqref{eq:hc:ad-reg}
    with~\eqref{eq:czybyszew} and noting $(( 1 + M^{-1} )^m - 1) \le 2^mM^{-1}$
    and $M > (2^{m + m^2} A^{2m} + 1) \sigma^{-1}$,
    \begin{gather*}
        \mu^{m}\bigl(X(a,\rho_i) \without ( Y(a,\rho_i) \cup W_{i-1}(a) ) \bigr)
        \ge \bigl( \unitmeasure{m} \rho_i^m \bigr)^m
        \bigl( A^m(\sigma - M^{-1}) - M^{-1} 2^{m+m^2}A^m \bigr) > 0 \,.
    \end{gather*}
    Observe that for $a \in S$, and $i = \N \without\{0\}$, and $j = 1,2,\ldots,m$,
    employing~\eqref{eq:P-mean-dist},
    \begin{multline*}
        \dist\bigl( h_{i,j}(a) - a, Q_{i-1}(a) \bigr)
        \le \biggl( M \fint_{\cball a{\rho_{i-1}}} \dist(z-a,Q_{i-1}(a))^p \ud \mu(z) \biggr)^{1/p}
        \\
        \le 2^{1+\alpha} M^{1/p} C_1 \kav^{l,p,\alpha}_{\mu}(a,\rho_{i-1})^{1/p} \rho_i^{1 + \alpha} \,.
    \end{multline*}
    Therefore, lemma~\ref{lem:planes} yields for $a \in S$ and $i \in \N
    \without\{0\}$
    \begin{gather*}
        \|\proj{Q_i(a)} - \proj{Q_{i-1}(a)} \| \le C_4 \kav^{l,p,\alpha}_{\mu}(a,\rho_{i-1})^{1/p} \rho_i^{\alpha} \,,
        \quad \text{where }
        C_4 = m \delta^{-1} 2^{1+\alpha} M^{1/p} C_1 \,.
    \end{gather*}
    Summing up a geometric series we see that for $a \in S$ the spaces $Q_i(a)$
    converge as~$i \to \infty$ to some $T(a) \in \grass nm$ satisfying
    \begin{gather}
        \label{eq:Q-conv-rate}
        \|\proj{Q_i(a)} - \proj{T(a)}\| \le C_5 \kav^{l,p,\alpha}_{\mu}(a,2\rho_i)^{1/p} \rho_i^{\alpha} \,,
        \quad \text{where }
        C_5 = (1 - 2^{-\alpha})^{-1} C_4 \,.
    \end{gather}
    Let $a \in S$, and $\rho \in (0,\infty)$, and $i \in \N$ be such that
    $\rho_{i+1} < \rho \le \rho_i \le r_0$. Then
    \begin{multline*}
        \biggl(
            \fint_{\cball a{\rho}} \dist(c-a,T(a))^p \ud \mu_{\Sigma}^1(c)
        \biggr)^{1/p}
        \le 
        \biggl(
            \fint_{\cball a{\rho}} \dist(c-a,Q_i(a))^p \ud \mu_{\Sigma}^1(c)
        \biggr)^{1/p}
        \\
        +
        \biggl(
            \fint_{\cball a{\rho}} \|Q_i(a) - T(a)\|^p |c-a|^p \ud \mu_{\Sigma}^1(c)
        \biggr)^{1/p}
        \\
        \le (C_1 + C_5) \kav^{l,p,\alpha}_{\mu}(a,2\rho_i)^{1/p} \rho_i^{1+\alpha}
        \le C_6 \kav^{l,p,\alpha}_{\mu}(a,4\rho)^{1/p} \rho^{1+\alpha} \,,
    \end{multline*}
    where $C_6 = 2^{1+\alpha} (C_1 + C_5)$.
\end{proof}

The following theorem~\ref{thm:allard-rect}, cited from~\cite{All72}, will allow
us to reduce our main theorem~\ref{thm:main} roughly to the case when $\mu =
\HM^m \restrict \Sigma$, where $\Sigma$ is a subset of a graph of a~$\cnt^1$
function.

\begin{thm}[\protect{cf.~\cite[2.8(5)]{All72}}]
    \label{thm:allard-rect}
    Suppose $\mu$ is a Radon measure over $\R^n$ such that for $\mu$ almost
    all~$a$ the following two conditions hold: $0 < \density^{m*}(\mu,a) <
    \infty$ and there exists $T \in \grass nm$ such that $\apTan{m}{\mu}{a}
    \subseteq T$. Then $\mu = \density^m(\mu,\cdot)\HM^m$ and $\R^n$~is
    countably $(\mu,m)$ rectifiable of class~$\cnt^1$.
\end{thm}

Now we are ready to prove the first part of the main theorem~\ref{thm:main}.

\begin{thm}
    \label{thm:rect}
    Suppose $\mu$ is a Radon measure which satisfies the density
    bounds~\eqref{eq:density-bounds} and $\kav^{l,p,\alpha}_{\mu}(a,1) < \infty$
    for $\mu$~almost all~$a$. Then $\R^n$ is countably $(\mu,m)$~rectifiable of
    class~$\cnt^{1,\alpha}$.
\end{thm}

\begin{proof}
    For $j \in \N \without \{0\}$ set
    \begin{gather*}
        A_j = \bigl\{ a \in \R^n
        : j^{-1} \unitmeasure{m} r^m < \mu(\cball ar) \le j \unitmeasure{m} r^m
        \text{ for } 0 < r < j^{-1} \bigr\} \,,
        \\
        A_0 = \varnothing \,,
        \quad
        B_j = A_j \without A_{j-1} \,,
        \quad 
        \mu_j = \mu \restrict B_j \,.
    \end{gather*}
    Since $(0,\infty) \ni r \mapsto \mu(\cball ar)$ is right-continuous for each
    $a \in \R^n$ and $\R^n \ni a \mapsto \mu(\cball ar)$ is upper
    semi-continuous for each $r \in (0,\infty)$ we deduce that $A_j$ are Borel
    sets. Clearly $A_j \subseteq A_{j+1}$ for $j \in \N$ so $\{ B_j : j \in \N
    \}$ is disjointed and, because $\mu$ satisfies~\eqref{eq:density-bounds}, we
    have
    \begin{displaymath}
        \mu\bigl(\R^n \without \tbcup\{A_j : j \in \N \}\bigr) = 0 \,;
    \end{displaymath}
    hence, it suffices to show that $\R^n$ is countably $(\mu_j,m)$~rectifiable
    of class~$\cnt^{1,\alpha}$ for each $j \in \N$. Fix $j \in \N \without
    \{0\}$. Define Borel sets $A_{j,0} = \varnothing$ and
    \begin{displaymath}
        A_{j,k} = \bigl\{ a \in A_j
        : (2j)^{-1} \unitmeasure{m} r^m < \mu_j(\cball ar) \le j \unitmeasure{m} r^m
        \text{ for } 0 < r < k^{-1} \bigr\}
    \end{displaymath}
    for $k \in \N \without \{0\}$. Observe that
    \begin{gather}
        \mu_j(\cball ar) \le j \unitmeasure{m} r^m
        \quad \text{for $a \in A_j$ and $0 < r \le j^{-1}$} \,,
        \\
        \density_*^m(\mu_j,a) \ge j^{-1} 
        \quad \text{for $\mu_j$ almost all $a$ by~\cite[2.9.11]{Fed69} since $B_j$ is Borel} \,;
        \\
        \label{eq:rect:Ajk-full}
        \text{thus,} \quad
        \mu_j\bigl(\R^n \without \tbcup\{ A_{j,k} : k \in \N \}\bigr) = 0 \,.
    \end{gather}
    For each $k \in \N$ and $a \in A_{j,k}$ we apply~\ref{cor:fat-simp-exist}
    with $\mu_j$, $k^{-1}$, $a$, $2^mj$ in place of $\mu$, $r_0$, $a$, $A$ to
    find out that there exists $\delta = \delta(n,m,j,k) \in (0,1]$ and $\sigma
    = \sigma(n,m,j,k) \in (0,1]$ such that
    \begin{displaymath}
        \mu_j^m(X_{\delta}(a,r)) \ge \sigma \mu_j(\cball ar) 
        \quad \text{for $0 < r < k^{-1}$} \,.
    \end{displaymath}
    Next, for each $k \in \N$ we apply~\ref{lem:height-control} with $\mu_j$,
    $\delta(n,m,k,j)$, $\sigma(n,m,k,j)$, $2j$, $k^{-1}$, $A_{j,k}$ in place of
    $\mu$, $\delta$, $\sigma$, $A$, $r_0$, $S$ to see that for each $k \in \N$
    and $a \in A_{j,k}$ there exists some $T(a) \in \grass nm$ such that
    \begin{equation}
        \label{eq:rect:tan-subs}
        \apTan{m}{\mu_j}{a} \subseteq T(a) \,,
    \end{equation}
    and, if $l < m+2$, then
    \begin{equation}
        \label{eq:rect:l<m+2}
        \frac{|\pproj{T(a)} (b-a)|}{|b-a|^{1 + \alpha}} \le C \kav^{l,p,\alpha}_{\mu}(a,|b-a|)^{1/p} 
        \quad \text{for $\mu_j$ almost all $b \in \cball a{k^{-1}}$} \,,
    \end{equation}
    and, if $l = m+2$, then, using H{\"o}lder's inequality,
    \begin{equation}
        \label{eq:rect:l=m+2}
        \fint_{\cball ar} \frac{|\pproj{T(a)}(c-a)|}{r^{1+\alpha}}  \ud \mu_j(c)
        \le C \kav^{l,p,\alpha}_{\mu}(a,4r)^{1/p} 
        \quad \text{for $0 < r < k^{-1}$} \,.
    \end{equation}
    Employing~\ref{thm:allard-rect} together with~\eqref{eq:rect:tan-subs}
    and~\eqref{eq:rect:Ajk-full} yields that $A_{j}$ is countably
    $(\mu_j,m)$~rectifiable of class~$\cnt^1$. Hence, there exists a countable
    family $\mathcal A$ of $\cnt^1$~submanifolds of~$\R^n$ such that 
    \begin{equation}
        \label{eq:rect:mu-j-repr}
        \mu_j = \density^m(\mu_j,\cdot) \HM^m \restrict \tbcup \mathcal A \,.
    \end{equation}
    Thus, to finish the proof it suffices to show that for each $k \in \N$ and
    $M \in \mathcal A$ the set $M \cap A_{j,k}$ is countably
    $(\HM^m,m)$~rectifiable of class~$\cnt^{1,\alpha}$.

    Fix $M \in \mathcal A$ and $k \in \N$ such that $\mu_j(M \cap A_{j,k}) >
    0$. Employing the definition of a~submanifold~\cite[3.1.19(5)]{Fed69},
    we~can represent~$M$ locally, around any $a \in M$, as a graph over the
    tangent plane~$\Tan Ma$ of~some $\cnt^1$ function, i.e., we can find
    a~neighborhood~$U_a$ of~$a$ in~$\R^n$ and projections $\pp_a \in \OP(n,m)$,
    $\qq_a \in \OP(n,n-m)$ such that
    \begin{gather*}
        \im \pp_a^\ast = \Tan Ma \,,
        \quad
        \im \qq_a^\ast = {\Tan Ma}^\perp \,,
        \quad
        \pp_a|_{M \cap U_a} \text{ is injective} \,,
        \\
        (\pp_a|_{M \cap U_a})^{-1} : \pp_a \lIm U_a \rIm \to \R^n \text{ is of class~$\cnt^1$} \,,
        \quad
        D\bigl((\pp_a|_{M \cap U_a})^{-1}\bigr)(\pp_a(a)) = \pp_a^\ast \,.
    \end{gather*}
    Set $F_a = (\pp_a|_{M \cap U_a})^{-1}$ and $f_a = \qq_a \circ F_a$; then
    \begin{displaymath}
        F_a = \pp_a^\ast + \qq_a^\ast \circ f_a \,,
        \quad
        f_a(\pp_a(a)) = 0 \,,
        \quad
        Df_a(\pp_a(a)) = 0 \,.
    \end{displaymath}
    Define an open ``cuboid'' adjusted to~$\Tan Ma$ of radius $r \in
    (0,\infty)$ by the formula
    \begin{displaymath}
        \cyl ar
        = \bigl\{ y \in \R^n : |\pp_a(y - a)| < r \text{ and } |\qq_a(y - a)| < r \bigr\} \,.
    \end{displaymath}
    Recall $U_a$ is a~neighborhood of~$a$ in~$\R^n$ so $a \in \interior U_a$.
    Thus, for all $a \in M$ there exists a~radius~$r_a > 0$ such that
    \begin{equation}
        \label{eq:rect:lip-bound}
        M \cap \cyl a{r_a} = F_a\lIm \oball{\pp_a(a)}{r_a} \rIm 
        \quad \text{and} \quad \|Df_a(x)\| \le \tfrac 12
        \quad \text{for $x \in \cball{\pp_a(a)}{r_a}$.}
    \end{equation}

    Next, observe that $M \cap A_{j,k}$ is a second-countable space as a
    subspace of a~second-countable space~$\R^n$; hence, it has the Lindel{\"o}f
    property (cf.~\cite[Theorem~30.3]{Mun00}). Thus, from the open covering
    $\bigl\{ \cyl a{r_a} : a \in M \cap A_{j,k} \bigr\}$ of~$M \cap A_{j,k}$,
    one can choose a countable subcovering $\bigl\{ \cyl{a_i}{r_{a_i}} : i \in
    \N \bigr\}$ of~$M \cap A_{j,k}$. Now it suffices to prove that $M \cap
    A_{j,k} \cap \cyl{a_i}{r_{a_i}}$ is countably $(\HM^m,m)$~rectifiable of
    class~$\cnt^{1,\alpha}$ for each $i \in \N$.

    Fix an $i \in \N$. Since $A_{j,k}$ is Borel it follows that $E =
    \pp_{a_i}\lIm M \cap A_{j,k} \cap \cyl{a_i}{r_{a_i}} \rIm$ is
    $\LM^m$~measurable by~\cite[2.2.13]{Fed69}. By~\eqref{eq:rect:lip-bound} we
    see that $\HM^m(F_{a_i}\lIm E \rIm) < \infty$. Clearly we also have $T(a_i)
    = \Tan M{a_i}$ because of~\eqref{eq:rect:tan-subs} and~$\dim T(a_i) = m$.
    Hence, we can apply~\ref{cor:main-rect-cond} with $\pp_{a_i}$, $\qq_{a_i}$,
    $T(a_i)$, $f_{a_i}$, $F_{a_i}$, $E$, $F_{a_i}\lIm E \rIm$ in place of $\pp$,
    $\qq$, $T$, $f$, $F$, $A$, $\Sigma$. With these substitutions,
    recalling~\ref{rem:r-limit} and~\eqref{eq:rect:mu-j-repr} and $j^{-1} \le
    \density^m(\mu_j,x) \le j$ for $\mu_j$ almost all $x \in A_j$, we see
    that~\eqref{eq:rect:l<m+2} implies~\eqref{eq:aplimsup-cond} in case $l <
    m+2$ and~\eqref{eq:rect:l=m+2} implies~\eqref{eq:limsup-int-cond} in case $l
    = m+2$. Consequently, we deduce that $M \cap A_{j,k} \cap
    \cyl{a_i}{r_{a_i}}$ is countably $(\HM^m,m)$~rectifiable of
    class~$\cnt^{1,\alpha}$.
\end{proof}

\section{Sharpness of the exponent $\alpha$ and other admissible integrands}
\label{sec:sharp}

Here we prove the second part of our main theorem~\ref{thm:main}. We also
consider a few different functions which can replace $\mlc$ in the definition of
$\kav^{l,p,\alpha}_{\mu}$.

\begin{defin}
    For $T = (a_0,\ldots,a_{m+1}) \in (\R^n)^{m+2}$ we set
    \begin{displaymath}
        \hmin(T) = \min\bigl\{ \dist\bigl( a_j, \aff( \{ a_0,\ldots,a_{m+1} \} \without \{a_j\} ) \bigr)
        : j = 0,1,\ldots,m+1 \bigr\} \,,
    \end{displaymath}
    where $\aff(A)$ denotes the smallest affine plane containing the set $A
    \subseteq \R^n$.
\end{defin}

\begin{defin}
    For $T = (a_0,\ldots,a_{m+1}) \in (\R^n)^{m+2}$ we set
    \begin{displaymath}
        \mlc_{\mathrm h}(T) = \frac{\hmin(T)}{\diam(\simp T)}
        \quad \text{if $\diam(\simp T) > 0$} \,,
        \quad
        \mlc_{\mathrm h}(T) = 0 \quad \text{otherwise} \,.
    \end{displaymath}
\end{defin}

\begin{rem}
    A suggestion to use $\mlc_{\mathrm{h}}$ instead of~$\mlc$ appeared
    in~\cite[\S10]{LW11}.
\end{rem}

For any $T \in (\R^n)^{m+2}$ to estimate $\hmin(T)$ it suffices to consider~$T$
inside the $(m+1)$~dimensional vector subspace of~$\R^n$ containing~$\simp T$.
Hence, the following estimate follows immediately from a result of Gritzmann and
Lassak~\cite[Lemma~3]{GL89}
\begin{lem}
    \label{lem:boxed-simp}
    Suppose $h \in [0,\infty)$, and $T \in (\R^n)^{m+2}$, and $a \in \R^n$, and
    $S \in \grass nm$, and $\simp T \subseteq \bigl\{
        b \in \R^n :
        |\pproj S (b-a)| \le h
        \bigr\}$.
    Then $\hmin(T) \le (m+2) h$.
\end{lem}

\begin{cor}
    \label{cor:graph-hmin-est}
    Let $\pp$, $\qq$, $T$, $A$, $f$, $F$ be as in~\ref{rem:pqF-setup}.  Suppose
    $\alpha \in [0,1]$, and $A = \R^m$, and $f$ is of class~$\cnt^{1,\alpha}$,
    and $\Sigma = \im F$, and $T = (a_0,\ldots,a_{m+1}) \in \Sigma^{m+2}$
    satisfy $\hmin(T) > 0$, and $d = \diam(\simp T)$, and
    \begin{displaymath}
        M = \sup\big\{
        \|Df(\pp(b) - Df(\pp(a_0))\| \cdot |\pp(b) - \pp(a_0)|^{-\alpha} 
        : b \in \Sigma \,,\, 0 < |b - a_0| \le d
        \big\} \,.
    \end{displaymath}
    Then $\hmin(T) \le M(m+2) d^{1+\alpha}$.
\end{cor}

\begin{proof}
    Note $|f(\pp(b) - f(\pp(a_0)) - Df(\pp(a_0))(\pp(b) - \pp(a_0))| \le M
    |\pp(b) - \pp(a_0)|^{1+\alpha} \le M |b - a_0|^{1+\alpha}$ for all $b \in
    \Sigma$ with $0 < |b - a_0| \le d$. Hence, employing~\eqref{eq:Taylor-dist},
    we obtain $\simp T \subseteq \bigl\{ b \in \R^n :
    |\pproj{\Tan{\Sigma}{a_0}}(b - a_0)| \le M d^{1+\alpha}
    \bigr\}$. Thus,~\ref{lem:boxed-simp} yields $\hmin(T) \le M (m+2)
    d^{1+\alpha}$.
\end{proof}

\begin{cor}
    \label{cor:finite-on-smooth}
    Let $\pp$, $\qq$, $T$, $A$, $f$, $F$ be as in~\ref{rem:pqF-setup}.  Suppose
    $\alpha,\beta \in [0,1]$, and $\alpha < \beta$, and $A = \R^m$, and $f$ is
    of class~$\cnt^{1,\beta}$, and $\Sigma = F\lIm \cball 01 \rIm$, and $p \in
    [1,\infty)$, and $l \in \{ 1,\ldots,m+2 \}$, and $\mu = \HM^m \restrict
    \Sigma$, and $\kav^{l,p,\alpha}_{h;\mu}$ is defined as
    $\kav^{l,p,\alpha}_{\mu}$ in~\eqref{eq:def-kav} but with $\mlc_{\mathrm h}$
    in place of~$\mlc$. Then there exists $\Gamma = \Gamma(m,l,f,p,\alpha,\beta)
    \in (0,\infty)$ such that
    \begin{gather*}
        \kav^{l,p,\alpha}_{h;\mu}(a,1) \le \Gamma 
        \quad \text{for all $a \in \R^n$} \,.
    \end{gather*}
\end{cor}

\begin{proof}
    Set
    \begin{gather*}
        M = \sup\big\{
        \|Df(x) - Df(y))\| \cdot |x-y|^{-\beta} 
        : x,y \in \cball 01
        \big\} \,,
        \\
        L = \sup\bigl\{ \|DF(x)\| : x \in \cball 01 \bigr\} \,.
    \end{gather*}
    If $p(\beta - \alpha) \ge m(l-1)$ (in particular if $l=1$), then
    by~\ref{cor:graph-hmin-est}
    \begin{displaymath}
        \kav^{l,p,\alpha}_{h;\mu}(a,1)
        \le M^p(m+2)^p 2^{p(\beta - \alpha) - m(l-1)} 
        \quad \text{for $a \in \R^n$}\,.
    \end{displaymath}

    Assume now $p(\beta - \alpha) < m(l-1)$ (in particular $l > 1$). For $a \in
    \R^n$ and $i \in \N$ define
    \begin{gather*}
        A_i(a) = \bigl\{ (a_1,\ldots,a_{l-1}) \in (\R^n)^{l-1} : 2^{-i} < \diam(\simp(a,a_1,\ldots,a_{l-1})) \le 2^{-i+1} \bigr\} \,.
    \end{gather*}
    Then $A_i(a) \subseteq {\cball a{2^{-i+1}}}^{l-1}$.
    Employing~\ref{cor:graph-hmin-est} and the area formula~\cite[3.2.3]{Fed69}
    we get
    \begin{multline*}
        \kav^{l,p,\alpha}_{h;\mu}(a,1)
        \le \sum_{i=0}^{\infty} \int_{A_i(a)}
        \frac{(M (m+2))^{p} \ud \mu^{l-1}(a_1,\ldots,a_{l-1})}{\diam(\{a,a_1,\ldots,a_{l-1}\})^{m(l-1) - p (\beta - \alpha)}}
        \\
        \le (M (m+2))^{p} \sum_{i=0}^{\infty} \mu(A_i(a)) 2^{i (m(l-1) - p(\beta - \alpha))}
        \\
        \le (M (m+2))^{p} (2L)^{m(l-1)} \unitmeasure{m}^{l-1} 
        \sum_{i=0}^{\infty} 2^{-i p(\beta - \alpha)} < \infty
        \qedhere
    \end{multline*}
\end{proof}

\begin{cor}
    \label{cor:hmin-sharp}
    Let $\mu$ be a Radon measure over $\R^n$ satisfying
    \eqref{eq:density-bounds}, and $l \in \{1,2,\ldots,m+2\}$, and $\alpha \in
    (0,1)$, and $p \in [1,\infty)$. Then for any $\varepsilon \in (0,1-\alpha)$
    there exists a Radon measure~$\mu$ satisfying~\eqref{eq:density-bounds}
    and~$\kav^{l,p,\alpha}_{h;\mu}(a) < \infty$ for $\mu$~almost all~$a$ and such
    that $\R^n$ is not countably $(\mu,m)$~rectifiable of
    class~$\cnt^{1,\alpha+\varepsilon}$.
\end{cor}

\begin{proof}
    Let $f : \R^m \cap \oball 02 \to \R^{n-m}$ be of class
    $\cnt^{1,\alpha+\varepsilon/2}$ such that the graph of~$f$ is not
    $(\HM^m,m)$~rectifiable of class~$\cnt^{1,\alpha+\varepsilon}$ in the sense
    of~\cite[Definition~3.1]{AS94} -- such function can be constructed
    using~\cite[Appendix]{AS94}. Let $\pp$, $\qq$, $T$, $F$ be related to~$f$ as
    in~\ref{rem:pqF-setup}. Set $\Sigma = F\lIm \cball 01 \rIm$ and $\mu = \HM^m
    \restrict \Sigma$. Clearly $\R^n$ is not countably $(\mu,m)$ rectifiable of
    class $\cnt^{1,\alpha+\varepsilon}$. However, since $f$ is of
    class~$\cnt^{1,\alpha+\varepsilon/2}$, we see that
    $\kav^{l,p,\alpha}_{h;\mu}(a) < \infty$ by~\ref{cor:finite-on-smooth}.
\end{proof}

\begin{defin}
    For $T = (a_0,a_1,\ldots,a_{m+1}) \in (\R^n)^{m+2}$ satisfying $\hmin(T) >
    0$ and $i \in I = \{0, 1, \ldots, m+1\}$ define
    \begin{gather*}
        p_m\sin_i(T) = \frac{|(a_1 - a_0) \wedge \cdots \wedge (a_{m+1} - a_0)|}{\prod_{j=0 \,, j \ne i}^{m+1}|a_j - a_i|} \,,
        \\
        \mlc_{\mathrm{min}}(T) = \min \{ p_m\sin_{i}(T) : i \in I \} \,,
        \quad
        \mlc_{\mathrm{max}}(T) = \max \{ p_m\sin_{i}(T) : i \in I \} \,,
        \\
        \mlc_{\mathrm{dls}}(T) = \inf \Bigl\{ \bigl( {\textstyle \sum_{i=0}^{m+1}} \dist(a_i, L)^2 \bigr)^{1/2} \diam(\simp T)^{-1}
        : \text{$L$ an affine $m$-plane in $\R^n$} \Bigr\} \,.
    \end{gather*}
    If $\hmin(T) = 0$, then set $\mlc_{\mathrm{min}}(T) = \mlc_{\mathrm{max}}(T) = \mlc_{\mathrm{dls}}(T) = 0$.
\end{defin}

\begin{rem}
    The definitions of $\mlc_{\mathrm{min}}$, $\mlc_{\mathrm{max}}$ are
    motivated by~\cite[\S6.1.1]{LW09} and the definition of
    $\mlc_{\mathrm{dls}}$ by~\cite[\S4]{LW12}
\end{rem}

\begin{lem}
    \label{lem:various}
    There exists $\Gamma = \Gamma(m) \in [1,\infty)$ such that for $T \in
    (\R^n)^{m+2}$ we have
    \begin{gather*}
        \mlc(T) \le \Gamma
        \min\bigl\{ \mlc_{\mathrm{min}}(T) \,,\, \mlc_{\mathrm{max}}(T) \,,\, \mlc_{\mathrm{dls}}(T) \,,\, \mlc_{\mathrm h}(T) \bigr\} 
        \\
        \text{and} \quad
        \max\bigl\{ \mlc_{\mathrm{min}}(T) \,,\, \mlc_{\mathrm{dls}}(T) \,,\, \mlc(T) \bigr\} 
        \le \Gamma \mlc_{\mathrm h}(T) \,.
    \end{gather*}
\end{lem}

\begin{proof}
    Let $T = (a_0,\ldots,a_{m+1}) \in(\R^n)^{m+2}$. If $\hmin(T) = 0$, then we
    get zero on both sides of both inequalities; thus, assume $\hmin(T) > 0$.
    Permuting the tuple $T$ we can assume $\hmin(T) = |\pproj{P}(a_{m+1}-a_0)|$,
    where $P = \lin\{a_1-a_0,\ldots,a_m-a_0\}$. Using the triangle inequality we
    can find $i \in \{0,1,\ldots,m\}$ such that $2 |a_{m+1} - a_{i}| \ge
    \diam(\simp T)$; thus, permuting the tuple $(a_0,\ldots,a_m)$, we can also
    assume $i = 0$. Then
    \begin{gather*}
        \mlc_{\mathrm{min}}(T) \le 
        p_m\sin_{i}(T) 
        = \frac{|(a_1 - a_{0}) \wedge \cdots \wedge (a_{m} - a_{0})| \cdot |\pproj{P}(a_{m+1}-a_0)|}
        {|a_1 - a_{0}| \cdots |a_{m} - a_{0}| \cdot |a_{m+1} - a_0|}
        \le 2 \mlc_{\mathrm h}(T) \,,
        \\
        \text{and} \quad
        \mlc(T) 
        = \frac{|(a_1 - a_{0}) \wedge \cdots \wedge (a_{m} - a_{0})| \cdot |\pproj{P}(a_{m+1}-a_0)|}
        {(m+1)!\, \diam(\simp T)^{m} \cdot \diam(\simp T)}
        \le \frac{\mlc_{\mathrm h}(T)}{(m+1)!}  \,,
        \\
        \text{and} \quad
        \mlc_{\mathrm h}(T)^2 =
        {\textstyle \sum_{i=0}^{m+1}} \dist(a_i, a_0 + P)^2 \,,
        \quad \text{so} \quad 
        \mlc_{\mathrm{dls}}(T) \le \mlc_{\mathrm h}(T) \,.
    \end{gather*}
    Hence, $\max\{ \mlc_{\mathrm{min}}(T), \mlc_{\mathrm{dls}}(T), \mlc(T) \}
    \le 2 \mlc_{\mathrm h}(T)$.

    Clearly $\mlc_{\mathrm{min}}(T) \le \mlc_{\mathrm{max}}(T)$. Since $|a_i -
    a_j| \le \diam(\simp T)$ for all $i,j \in \{0,1,\ldots,m+1\}$ it is also
    clear that $\mlc_{\mathrm{min}}(T) \ge (m+1)!\,\mlc(T)$.
    From~\cite[(A.2)]{LW11} we farther deduce $\mlc_{\mathrm{h}}(T) \le (m+2)
    \mlc_{\mathrm{dls}}(T)$. Therefore, $\mlc(T) \le (m+2) \min\{
    \mlc_{\mathrm{min}}(T) , \mlc_{\mathrm{max}}(T) , \mlc_{\mathrm{dls}}(T) ,
    \mlc_{\mathrm h}(T)\}$.
\end{proof}

\begin{cor}
    \label{cor:main}
    Let $\mu$ be a Radon measure over $\R^n$ satisfying
    \eqref{eq:density-bounds}, and $l \in \{1,2,\ldots,m+2\}$, and $\alpha \in
    (0,1]$, and $p \in [1,\infty)$, and $\kav^{l,p,\alpha}_{*;\mu}$ be defined
    as $\kav^{l,p,\alpha}_{\mu}$ in~\eqref{eq:def-kav} but with~$\mlc$ replaced
    by one of $\mlc$, $\mlc_{\mathrm{min}}$, $\mlc_{\mathrm{max}}$,
    $\mlc_{\mathrm{dls}}$, $\mlc_{\mathrm h}$. Assume
    $\kav^{l,p,\alpha}_{*;\mu}(a,1) < \infty$ for $\mu$~almost all~$a$. Then
    $\R^n$~is countably $(\mu,m)$~rectifiable of class~$\cnt^{1,\alpha}$ and
    $\mu$ is absolutely continuous with respect to~$\HM^m$.

    Moreover, if $\mlc$ is replaced by one of $\mlc$, $\mlc_{\mathrm{min}}$,
    $\mlc_{\mathrm{dls}}$, $\mlc_{\mathrm h}$ and $\alpha < 1$, then for any
    $\varepsilon \in (0,1-\alpha)$ there exists a measure $\mu$
    satisfying~\eqref{eq:density-bounds} and $\kav^{l,p,\alpha}_{*;\mu}(a) <
    \infty$ for $\mu$~almost all~$a$ and such that $\R^n$~is not countably
    $(\mu,m)$~rectifiable of class~$\cnt^{1,\alpha+\varepsilon}$.
\end{cor}

\begin{proof}
    The claim readily follows from~\ref{thm:rect} and~\ref{cor:hmin-sharp}
    and~\ref{lem:various}.
\end{proof}

\begin{rem}
    The author does not know whether the second part of~\ref{cor:main} holds if
    one uses $\mlc_{\mathrm{max}}$ in place of $\mlc$.
\end{rem}

\section*{Acknowledgements}
The author was partially supported by the Foundation for Polish Science and
partially by NCN Grant no. 2013/10/M/ST1/00416 while he was on leave from
Institute of Mathematics of the University of Warsaw. 

The author is also indebted to Ulrich Menne for many fruitful discussions.

{\linespread{0.2}
  \footnotesize
  \bibliography{refs}{}

\def\cprime{$'$} \def\cprime{$'$} \def\cprime{$'$}
\begin{thebibliography}{KSvdM13}

\bibitem[Alb94]{Alb94}
Giovanni Alberti.
\newblock On the structure of singular sets of convex functions.
\newblock {\em Calc. Var. Partial Differential Equations}, 2(1):17--27, 1994.
\newblock URL: \url{http://dx.doi.org/10.1007/BF01234313}.

\bibitem[Ale39]{Ale39}
A.~D. Alexandroff.
\newblock Almost everywhere existence of the second differential of a convex
  function and some properties of convex surfaces connected with it.
\newblock {\em Leningrad State Univ. Annals [Uchenye Zapiski] Math. Ser.},
  6:3--35, 1939.

\bibitem[All72]{All72}
William~K. Allard.
\newblock On the first variation of a varifold.
\newblock {\em Ann. of Math. (2)}, 95:417--491, 1972.

\bibitem[AS94]{AS94}
Gabriele Anzellotti and Raul Serapioni.
\newblock {$\mathcal{C}^k$}-rectifiable sets.
\newblock {\em J. Reine Angew. Math.}, 453:1--20, 1994.

\bibitem[AT15]{AT15}
Jonas Azzam and Xavier Tolsa.
\newblock Characterization of n-rectifiability in terms of {J}ones' square
  function: {P}art {II}.
\newblock {\em Geom. Funct. Anal.}, 25(5):1371--1412, 2015.
\newblock URL: \url{http://dx.doi.org/10.1007/s00039-015-0334-7}.

\bibitem[BK12]{BK12}
Simon Blatt and S{\l}awomir Kolasi{\'n}ski.
\newblock Sharp boundedness and regularizing effects of the integral {M}enger
  curvature for submanifolds.
\newblock {\em Adv. Math.}, 230(3):839--852, 2012.
\newblock URL: \url{http://dx.doi.org/10.1016/j.aim.2012.03.007}.

\bibitem[CZ61]{CZ61}
A.-P. Calder{\'o}n and A.~Zygmund.
\newblock Local properties of solutions of elliptic partial differential
  equations.
\newblock {\em Studia Math.}, 20:171--225, 1961.

\bibitem[Dav98]{Dav98}
Guy David.
\newblock Unrectifiable {$1$}-sets have vanishing analytic capacity.
\newblock {\em Rev. Mat. Iberoamericana}, 14(2):369--479, 1998.
\newblock URL: \url{http://dx.doi.org/10.4171/RMI/242}.

\bibitem[DS91]{DS91}
Guy {David} and Stephen {Semmes}.
\newblock {\em {Singular integrals and rectifiable sets in $R\sp n$. Au-del\`a
  des graphes lipschitziens.}}
\newblock Montrouge: Soci\'et\'e Math\'ematique de France, 1991.

\bibitem[DS93]{DS93}
Guy David and Stephen Semmes.
\newblock {\em Analysis of and on uniformly rectifiable sets}, volume~38 of
  {\em Mathematical Surveys and Monographs}.
\newblock American Mathematical Society, Providence, RI, 1993.
\newblock URL: \url{http://dx.doi.org/10.1090/surv/038}.

\bibitem[Fed69]{Fed69}
Herbert Federer.
\newblock {\em Geometric measure theory}.
\newblock Die Grundlehren der mathematischen Wissenschaften, Band 153.
  Springer-Verlag New York Inc., New York, 1969.

\bibitem[Fu11]{Fu11}
Joseph H.~G. Fu.
\newblock An extension of {A}lexandrov's theorem on second derivatives of
  convex functions.
\newblock {\em Adv. Math.}, 228(4):2258--2267, 2011.
\newblock URL: \url{http://dx.doi.org/10.1016/j.aim.2011.07.002}.

\bibitem[GL89]{GL89}
Peter {Gritzmann} and Marek {Lassak}.
\newblock {Estimates for the minimal width of polytopes inscribed in convex
  bodies.}
\newblock {\em {Discrete Comput. Geom.}}, 4(6):627--635, 1989.

\bibitem[KM15]{KM15}
S.~{Kolasi{\'n}ski} and U.~{Menne}.
\newblock {Decay rates for the quadratic and super-quadratic tilt-excess of
  integral varifolds}.
\newblock {\em ArXiv e-prints}, January 2015.
\newblock \href {http://arxiv.org/abs/1501.07037} {\path{arXiv:1501.07037}}.

\bibitem[Kol15]{Kol15a}
S{\l}awomir Kolasi{\'n}ski.
\newblock Geometric {S}obolev-like embedding using high-dimensional
  {M}enger-like curvature.
\newblock {\em Trans. Amer. Math. Soc.}, 367(2):775--811, 2015.
\newblock URL: \url{http://dx.doi.org/10.1090/S0002-9947-2014-05989-8}.

\bibitem[KS13]{KolSzum}
S{\l}awomir Kolasi{\'n}ski and Marta Szuma{\'n}ska.
\newblock Minimal {H}\"older regularity implying finiteness of integral
  {M}enger curvature.
\newblock {\em Manuscripta Math.}, 141(1-2):125--147, 2013.
\newblock URL: \url{http://dx.doi.org/10.1007/s00229-012-0565-y}.

\bibitem[KSv15]{KSM15}
S.~{Kolasi{\'n}ski}, P.~{Strzelecki}, and H.~{von der Mosel}.
\newblock {Compactness and isotopy finiteness for submanifolds with uniformly
  bounded geometric curvature energies}.
\newblock {\em ArXiv e-prints}, April 2015.
\newblock \href {http://arxiv.org/abs/1504.04538} {\path{arXiv:1504.04538}}.

\bibitem[KSvdM13]{KSM13}
S{\l}awomir Kolasi{\'n}ski, Pawe{\l} Strzelecki, and Heiko von~der Mosel.
\newblock Characterizing {$W^{2,p}$} submanifolds by {$p$}-integrability of
  global curvatures.
\newblock {\em Geom. Funct. Anal.}, 23(3):937--984, 2013.
\newblock URL: \url{http://dx.doi.org/10.1007/s00039-013-0222-y}.

\bibitem[L{\'e}g99]{Leg99}
J.~C. L{\'e}ger.
\newblock Menger curvature and rectifiability.
\newblock {\em Ann. of Math. (2)}, 149(3):831--869, 1999.
\newblock URL: \url{http://dx.doi.org/10.2307/121074}.

\bibitem[LL13]{LL13}
Chun-Liang Lin and Fon-Che Liu.
\newblock Approximate differentiability according to
  {S}tepanoff-{W}hitney-{F}ederer.
\newblock {\em Indiana Univ. Math. J.}, 62(3):855--868, 2013.
\newblock URL: \url{http://dx.doi.org/10.1512/iumj.2013.62.5024}.

\bibitem[LT94]{LT94}
Fon~Che Liu and Wei~Shyan Tai.
\newblock Approximate {T}aylor polynomials and differentiation of functions.
\newblock {\em Topol. Methods Nonlinear Anal.}, 3(1):189--196, 1994.

\bibitem[LW09]{LW09}
Gilad Lerman and J.~Tyler Whitehouse.
\newblock High-dimensional {M}enger-type curvatures. {P}art {II}:
  {$d$}-separation and a menagerie of curvatures.
\newblock {\em Constr. Approx.}, 30(3):325--360, 2009.
\newblock URL: \url{http://dx.doi.org/10.1007/s00365-009-9073-z}.

\bibitem[LW11]{LW11}
Gilad Lerman and J.~Tyler Whitehouse.
\newblock High-dimensional {M}enger-type curvatures. {P}art {I}: {G}eometric
  multipoles and multiscale inequalities.
\newblock {\em Rev. Mat. Iberoam.}, 27(2):493--555, 2011.
\newblock URL: \url{http://dx.doi.org/10.4171/RMI/645}.

\bibitem[LW12]{LW12}
Gilad Lerman and J.~Tyler Whitehouse.
\newblock Least squares approximations of measures via geometric condition
  numbers.
\newblock {\em Mathematika}, 58(1):45--70, 2012.
\newblock URL: \url{http://dx.doi.org/10.1112/S0025579311001720}.

\bibitem[Men09]{Men09}
Ulrich Menne.
\newblock Some applications of the isoperimetric inequality for integral
  varifolds.
\newblock {\em Adv. Calc. Var.}, 2(3):247--269, 2009.
\newblock URL: \url{http://dx.doi.org/10.1515/ACV.2009.010}.

\bibitem[Men10]{Men10}
Ulrich Menne.
\newblock A {S}obolev {P}oincar\'e type inequality for integral varifolds.
\newblock {\em Calc. Var. Partial Differential Equations}, 38(3-4):369--408,
  2010.
\newblock URL: \url{http://dx.doi.org/10.1007/s00526-009-0291-9}.

\bibitem[Men11]{Men11}
Ulrich Menne.
\newblock Second order rectifiability of integral varifolds of locally bounded
  first variation.
\newblock {\em J. Geom. Anal.}, 2011.
\newblock URL: \url{http://dx.doi.org/10.1007/s12220-011-9261-5}.

\bibitem[Men12]{Men12}
Ulrich Menne.
\newblock Decay estimates for the quadratic tilt-excess of integral varifolds.
\newblock {\em Arch. Ration. Mech. Anal.}, 204(1):1--83, 2012.
\newblock URL: \url{http://dx.doi.org/10.1007/s00205-011-0468-1}.

\bibitem[{Meu}15]{Meu15a}
M.~{Meurer}.
\newblock {Integral Menger Curvature and Rectifiability of $n$-dimensional
  Borel sets in Euclidean $N$-space}.
\newblock {\em ArXiv e-prints}, October 2015.
\newblock \href {http://arxiv.org/abs/1510.04523} {\path{arXiv:1510.04523}}.

\bibitem[Mun00]{Mun00}
James~R. Munkres.
\newblock {\em Topology}.
\newblock Prentice-Hall, Inc., Englewood Cliffs, N.J., second edition, 2000.

\bibitem[Re{\v{s}}68]{Res68}
Ju.~G. Re{\v{s}}etnjak.
\newblock Generalized derivatives and differentiability almost everywhere.
\newblock {\em Mat. Sb. (N.S.)}, 75(117):323--334, 1968.

\bibitem[Sch09]{Sch09}
Reiner Sch{\"a}tzle.
\newblock Lower semicontinuity of the {W}illmore functional for currents.
\newblock {\em J. Differential Geom.}, 81(2):437--456, 2009.
\newblock URL:
  \url{http://projecteuclid.org/getRecord?id=euclid.jdg/1231856266}.

\bibitem[Sim83]{Sim83}
Leon Simon.
\newblock {\em Lectures on geometric measure theory}, volume~3 of {\em
  Proceedings of the Centre for Mathematical Analysis, Australian National
  University}.
\newblock Australian National University, Centre for Mathematical Analysis,
  Canberra, 1983.

\bibitem[SvdM11]{SvdM11a}
Pawe{\l} Strzelecki and Heiko von~der Mosel.
\newblock Integral {M}enger curvature for surfaces.
\newblock {\em Adv. Math.}, 226(3):2233--2304, 2011.
\newblock URL: \url{http://dx.doi.org/10.1016/j.aim.2010.09.016}.

\bibitem[Tol15]{Tol15}
Xavier Tolsa.
\newblock Characterization of n-rectifiability in terms of {J}ones' square
  function: part {I}.
\newblock {\em Calc. Var. Partial Differential Equations}, 54(4):3643--3665,
  2015.
\newblock URL: \url{http://dx.doi.org/10.1007/s00526-015-0917-z}.

\end{thebibliography}
  \bibliographystyle{myalphaurl}
}

\end{document}